\documentclass[11pt]{amsart}

\usepackage{amsmath,amssymb,amsthm,multicol,mathtools,hyperref,dsfont,pinlabel,enumitem}
\usepackage[normalem]{ulem}
\usepackage[usenames,dvipsnames]{color}
\usepackage{mathrsfs}

\newcommand\cut{\setminus\!\setminus}

\newcommand\wh[1]{\scalebox{.9}{$\widehat{
#1
}$}}

\newcommand\Z{\mathbb{Z}}

\newcommand\R{\mathbb{R}}

\newcommand\bbm{\begin{bmatrix}}
\newcommand\ebm{\end{bmatrix}}
\newcommand\red[1]{\color{red}#1\color{black}}
\newcommand\FG[1]{\color{ForestGreen}#1\color{black}}
\newcommand\violet[1]{\color{Violet}#1\color{black}}

\newcommand\brown[1]{\color{Brown}#1\color{black}}

\newcommand\bs{\boldsymbol}

\theoremstyle{plain}
\newtheorem{theorem}{Theorem}[section]

\newtheorem{obs}[theorem]{Observation}
\newtheorem{prop}[theorem]{Proposition}
\newtheorem{cor}[theorem]{Corollary}
\newtheorem{correspondence}[theorem]{Correspondence}
\newtheorem{conjecture}[theorem]{Conjecture}
\newtheorem{fact}[theorem]{Fact}
\newtheorem*{T:ObviouslyPrimeEtc}{Theorem \ref{T:ObviouslyPrimeEtc}}
\newtheorem*{C:Diagrams}{Correspondence \ref{C:Diagrams}}
\newtheorem*{C:Links2}{Correspondence \ref{C:Links2}}
\newtheorem*{T:KupD}{Theorem \ref{T:KupD}}
\newtheorem*{T:VSplitEtc}{Theorem \ref{T:VSplitEtc}}
\newtheorem*{T:01}{Theorem \ref{T:01}}

\theoremstyle{definition}
\newtheorem{convention}[theorem]{Convention}

\newtheorem{construction}[theorem]{Construction}
\newtheorem{notation}[theorem]{Notation}
\newtheorem{definition}[theorem]{Definition}
\newtheorem{question}[theorem]{Question}
\newtheorem{problem}[theorem]{Problem}
\newtheorem{example}[theorem]{Example}
\newtheorem{move}{Move}

\theoremstyle{remark}
\newtheorem{rem}{Remark}

\begin{document}

\title[Prime virtual links]{Primeness of alternating virtual links}

\author{Thomas Kindred}

\address{Department of Mathematics \& Statistics, Wake Forest University \\
Winston-Salem North Carolina, 27109} 

\email{thomas.kindred@wfu.edu}
\urladdr{www.thomaskindred.com}



\maketitle

\begin{abstract}
Using a new tool called lassos, we establish a new correspondence between cellular link {diagrams} on closed surfaces and equivalence classes of virtual link {diagrams}. This is analogous to a well-known correspondence among the links represented by these diagrams, but with a crucial subtlety. 
We explain how, under these correspondences, the traditional notion of primeness for virtual links is stricter than the one for links in thickened surfaces. 
We extend a classical result of Menasco by proving 
that an alternating link in a thickened surface is prime in the stricter sense unless it is ``obviously" composite. (Adams et al and Howie--Purcell previously extended Menasco's result for the other notion of primeness.)  We describe, given an alternating virtual link diagram, how to determine 
by inspection whether the virtual link it represents is prime in either sense. 
\end{abstract}

\section{Introduction}\label{S:intro}


Traditionally, a nonclassical virtual link $K$ is said to be prime if 
it cannot be decomposed as a nontrivial connect sum. 
On the other hand, a nonstabilized\footnote{See \textsection\ref{S:Thick} for definitions of terms including {\it nonstabilized} and {\it cellular}.}  link $L$ in a thickened surface $\Sigma\times I$ of positive genus is traditionally called prime if, for any pairwise connect sum decomposition
   $(\Sigma\times I,L)=(\Sigma\times I,L_1)\#(S^3,L_2)$, $L_2$ is trivial. 
   
Interestingly, 
under a well-known correspondence between virtual links and nonstabilized links in thickened surfaces
, prime links in one setting do not always correspond to prime links in the other. For example, in Figure \ref{Fi:PrimesEx}, the knot $K$ in a thickened surface of genus 2 is prime, but
the corresponding virtual knot $L$ is not.
We show:

\begin{T:01}
Given a nonsplit virtual link $K$ and the corresponding nonstabilized link $L$ in a thickened surface $\Sigma\times I$:
\begin{enumerate}
\item $(\Sigma,L)$ is prime (in the traditional sense, which we call ``local'') if and only if $K$ admits no nontrivial connect sum decomposition $K=K_1\#K_2$ in which $K_1$ is a classical link and $g(K_2)=g(K)$; and
\item $K$ is prime if and only if $(\Sigma,L)$ is what we call ``pairwise prime'' (see Definition \ref{D:LPPrime}).
\end{enumerate}
\end{T:01}

The pairwise prime condition in part (2) is how Matveev {\it defines} prime virtual links \cite{mat}, but he does not 
mention (or prove) that this condition coincides with the natural definition in this paper's first sentence; part (2) of Theorem \ref{T:01} confirms the equivalence of these definitions.

\begin{figure}
\begin{center}
\hfill
\labellist\small
\pinlabel $K\!:$ at -50 150
\endlabellist
\includegraphics[width=.4\textwidth]{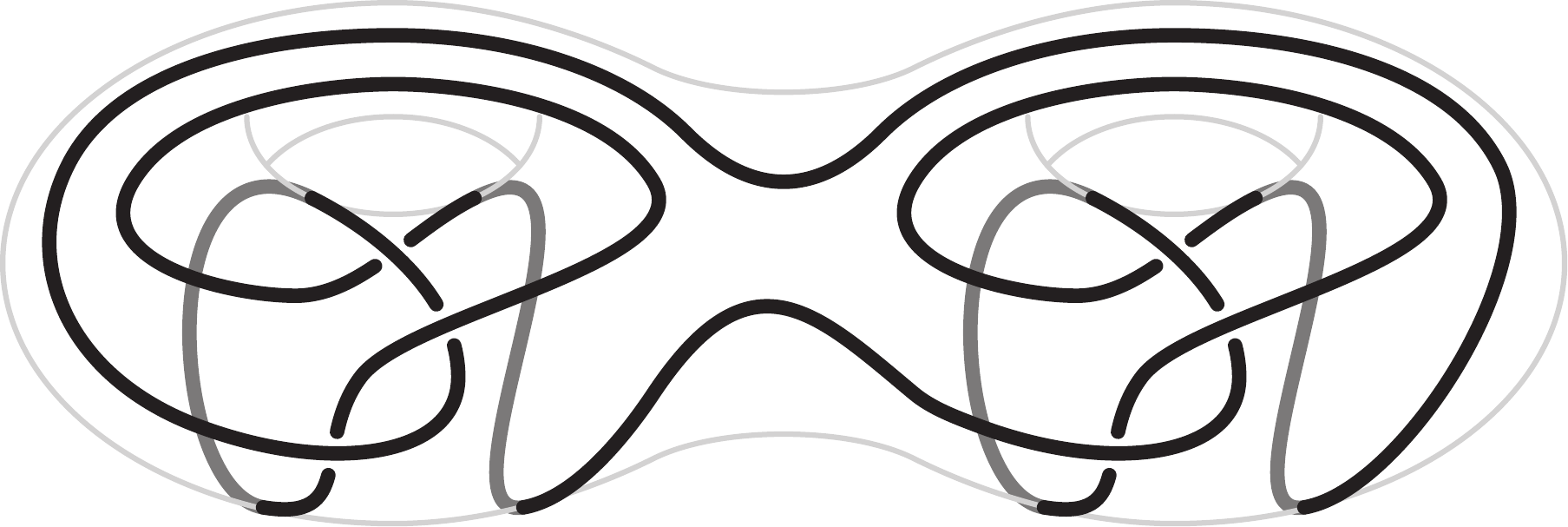}\hfill
\labellist\small
\pinlabel $L\!:$ at -50 150
\endlabellist
\includegraphics[width=.4\textwidth]{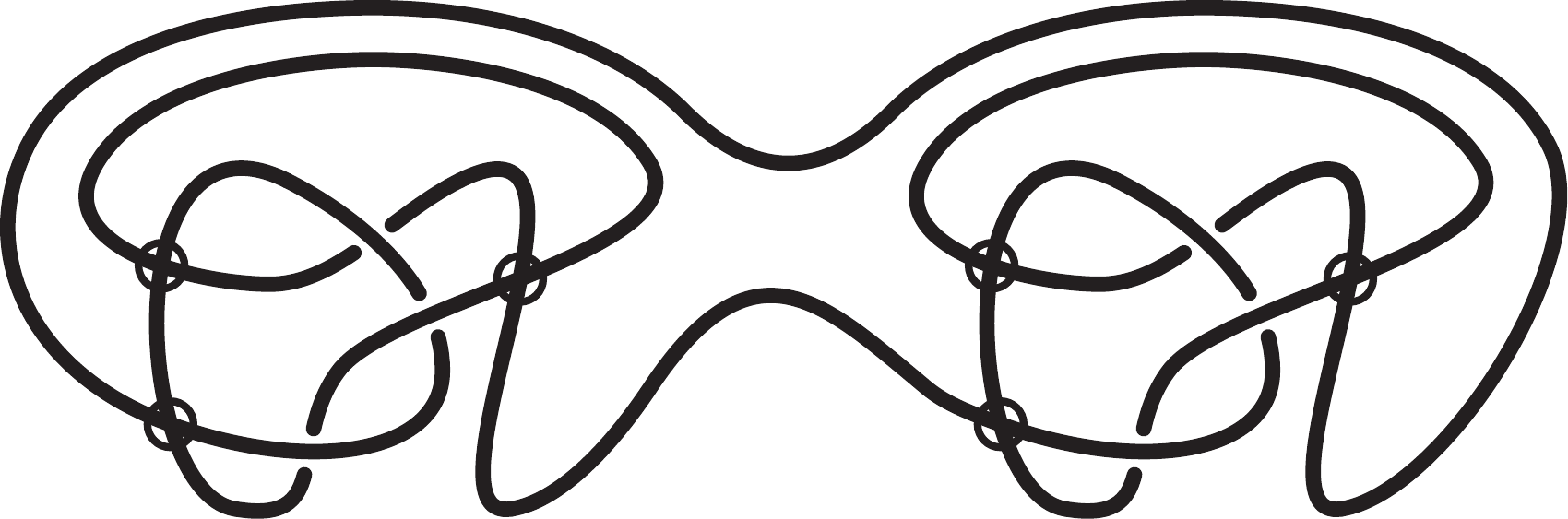}
\hfill
\caption{A knot that is locally, but not pairwise, prime.}
\label{Fi:PrimesEx}
\end{center}
\end{figure}   



Given a reduced cellular alternating diagram $D\subset\Sigma$ of a link $L\subset\Sigma\times I$, one can determine by inspecting $D$ whether or not $L$ is prime in either sense.  For local primeness, this was already known, by work of Menasco 
and Adams et al \cite{men84,adamsetal}, and Howie--Purcell proved a stronger theorem about local primeness in \cite{hp20}. We prove this for {\it pairwise} primeness:

\begin{T:ObviouslyPrimeEtc}
Given a reduced cellular
alternating diagram $D\subset \Sigma$ of a link $L\subset \Sigma\times I$, $(\Sigma,L)$ is pairwise
prime if and only if whenever $(\Sigma,D)=(\Sigma_1,D_1)\# (\Sigma_2,D_2)$, either $(\Sigma_1,D_1)$ or $(\Sigma_2,D_2)$ is $(S^2,\bigcirc)$.
\end{T:ObviouslyPrimeEtc}

 

To translate Theorem \ref{T:ObviouslyPrimeEtc} to a statement about virtual links and their diagrams, we use a well-known correspondence between virtual links and stable equivalence classes of links in thickened surfaces  \cite{kauff98,kaka,cks02}, together with a new correspondence between the associated {\it diagrams}:

\begin{C:Diagrams}
The following gives a bijection from equivalence classes $[V]$ of nonsplit virtual link diagrams under non-classical R-moves to cellular link diagrams on connected closed surfaces:

Choose $V\in[V]$, take a regular neighborhood $\nu V$ of $V$ in $S^2$, %
modify $\nu V$ near each virtual crossing of $V$ as shown in Figure \ref{Fi:VtoA},%
\footnote{At this intermediate stage, we have an {\it abstract link diagram}, which 
we will not need again.} 
and cap off each boundary component (abstractly) with a disk.
\end{C:Diagrams}

\begin{figure}
\begin{center}
\includegraphics[width=\textwidth]{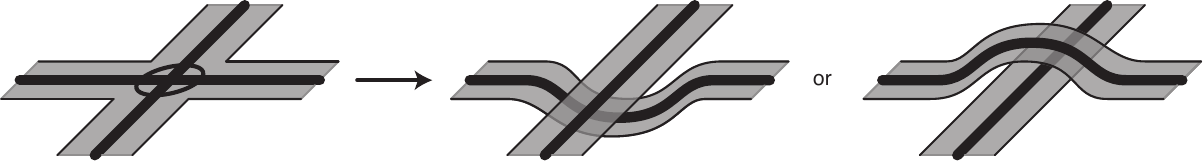}
\caption{Converting the neighborhood of a virtual link diagram to an abstract link diagram}
\label{Fi:VtoA}
\end{center}
\end{figure}

At first glance, this follows exactly the sort of construction described in \cite{kauff98,kaka,cks02}, but the reverse direction of the correspondence contains a hidden subtlety, one which is fundamental to understanding both correspondences and which, to the author's knowledge, has not previously been observed in the literature. We describe this subtlety in \textsection\ref{S:Subtle}.
%


Theorem \ref{T:ObviouslyPrimeEtc} and the related results of Adams et al \cite{adamsetal}, Howie--Purcell \cite{hp20}, and Menasco \cite{men84} state that an alternating link (in $S^3$ or a thickened surface) is composite (in an appropriate sense) 
if and only if it is ``obviously so'' in a given reduced alternating diagram. Theorem \ref{T:ObviouslyPrimeEtc} can be stated in the same manner.  When translating Theorem \ref{T:ObviouslyPrimeEtc} and the related results of Adams et al and Howie--Purcell to virtual link diagrams, however, ``obvious" feels inaccurate.  What does it mean for an alternating virtual link diagram $V$ to be ``obviously" composite (in either the local or pairwise sense)? 
Certainly, if $V$ decomposes as a diagrammatic connect sum of two nontrivial links, then it is obviously composite; but this is too restrictive (the theorem is untrue with such a strong requirement).  
Our remedy is (to consider not just $V$ but $[V]$ and) use a new tool called {\it lassos} to capture the salient features of $[V]$.  

Given a link diagram $D$ on a closed surface $\Sigma$, a lasso is a disk $X\subset \Sigma$ that contains all crossings of $D$. Similarly, given a virtual link diagram $V\subset S^2$, a lasso is a disk $X\subset S^2$ that contains all classical crossings of $V$ and no virtual ones. In both contexts, a lasso $X$ is {\it acceptable} if the part of the diagram in $X$ is connected and the part of the diagram outside $X$ does not admit an ``obvious" simplification. In \textsection\textsection\ref{S:lassoSigma}-\ref{S:lassoV}, we define these terms carefully and establish basic properties, like the fact that $D\subset\Sigma$ admits has an acceptable lasso if and only if $D$ is connected. In \textsection\ref{S:LD}, we introduce {\it lasso diagrams} and {\it lasso numbers}.  A lasso diagram is essentially a virtual link diagram in which the virtual crossings are captured combinatorially rather than shown  (see Figure \ref{Fi:Table}).  The lasso number is an invariant of virtual links.  We establish some basic properties.  Computing this invariant seems like a challenging, but approachable, problem.

In \textsection\ref{S:Corr}, we use lassos to establish Correspondence \ref{C:Diagrams}, which we then extend to equivalence classes of lasso diagrams under two types of moves (see Moves \ref{M:LDMove1}-\ref{M:LDMove2} and Figure \ref{Fi:Moves}). This gives the following extension of the well-known Correspondence \ref{C:Links}:

\begin{C:Links2}\label{C:Links2}
There is a triple bijective correspondence between (i) virtual links, (ii) stable equivalence classes of links in thickened surfaces, and (iii) equivalence classes of lasso diagrams under Moves \ref{M:LDMove1}-\ref{M:LDMove2} and classical R-moves.  
\end{C:Links2}

We also prove the following diagrammatic extension of Kuperberg's theorem:

\begin{T:KupD}
All minimal genus diagrams of a nonsplit virtual link are related by minimal-genus-preserving generalized R-moves.
\end{T:KupD}

Section \ref{S:ThickPrime} addresses local and pairwise primeness for links in thickened surfaces, culminating with Theorem \ref{T:ObviouslyPrimeEtc}, and \textsection\ref{S:Virtual} adapts that theorem to virtual link diagrams and lasso diagrams. Here, we find that nugatory crossings are always obvious in acceptable lasso diagrams, even though nugatory crossings can be harder to identify in virtual link diagrams.  This is particularly important because, in the virtual setting there are two types of nugatory crossings (removable and non-removable), and such crossings provide the main technical obstacle to translating statements about diagrammatic primeness (local or pairwise) to statements about prime virtual links.  Everything connects in our final result:

\begin{T:VSplitEtc}
Let $V'$ be a connected alternating 
diagram of a virtual link $K$, and consider a diagram $V\in[V']$ which 
admits an acceptable lasso $X$. Assume that $V$ has at least one virtual crossing. Then:
\begin{enumerate}[label=(\arabic*)]
\item When $V'$ has no nugatory crossings, $K$ is prime if and only if $V'$ is prime.
\item When $V'$ has no removably nugatory crossings, $K$ is locally prime if and only if $V'$ is locally prime.
\end{enumerate}
\end{T:VSplitEtc}

Section \ref{S:Conclude} offers some concluding thoughts.

\section{Background}

\subsection{Notation}\label{S:Notation}
Throughout:
\begin{itemize}
\item We work in the piecewise-linear category.
\item $\Sigma$ denotes a closed orientable surface, not necessarily connected or of positive genus. 
\item $I$ and $I_+$ respectively denote the intervals $[-1,1]$ and $[0,1]$.
\item In $\Sigma\times I$, we identify $\Sigma$ with $\Sigma\times\{0\}$ and write $\Sigma\times\{\pm1\}=\Sigma_\pm$.  
\item $\pi_\Sigma$ denotes projection $\pi_\Sigma:\Sigma\times I\to \Sigma$.
\item 
For a pair $(\Sigma,L)$, $L$ is a link in $\Sigma\times I$ which intersects each component of $\Sigma\times I$.
\item For a pair $(\Sigma,D)$, $D$ is a link diagram on $\Sigma$ which intersects each component of $\Sigma$. 
%
\item 
$\wh{S^3}$ denotes $S^3\setminus(\text{2 points})$, which is identified homeomorphically with $S^2\times \R$, and $\pi:\wh{S^3}\to S^2$ denotes projection.
\item 
$|X|$ denotes the number of connected components of $X$.
\item Given transverse submanifolds $S,T$ of some ambient manifold, the notations $|S\cap T|$ and $|S\pitchfork T|$ carry the same meaning; we use the latter notation if we wish to emphasize or clarify that $S$ and $T$ are transverse. 
\item In a manifold $X$, given a subset $Y$ that has a closed regular neighborhood, we denote this neighborhood $\nu Y$ and its interior $\overset{\circ}{\nu}Y$. Further, $X\cut Y$ denotes ``$X$ cut along $Y$," which is the metric closure of $X\setminus Y$.\footnote{$X\cut Y$ is homeomorphic to $X\setminus\overset{\circ}{\nu}Y$ but may have extra structure from $Y$ encoded in its boundary.  For example, if $F$ is a compact orientable surface in $S^3$, then $S^3\cut F$ is a sutured manifold: the extra structure here is the copy of $\partial F$ on $\partial(S^3\cut F)$, which cuts $\partial(S^3\cut F)$ into two copies of $F$. Similarly, if $(\Sigma,D)$ is cellular, then the boundary of each disk of $\Sigma\cut D$ contains a copy of each incident edge and vertex (i.e. crossing) from $D$.}
\end{itemize}

\subsection{Links in thickened surfaces}\label{S:Thick}

A pair $(\Sigma,L)$ is {\bf stabilized} if, for some circle%
\footnote{We use ``circle" as shorthand for ``smooth simple closed curve."} 
 $\gamma\subset \Sigma$, $L$ can be isotoped so that it is disjoint from the annulus $\gamma\times I$ but intersects each component of $(\Sigma\times I)\setminus(\gamma\times I)$; one can then {\it destabilize} the pair $(\Sigma,L)$ by cutting $\Sigma\times I$ along $\gamma\times I$ and attaching two 3-dimensional 2-handles in the natural way (this may disconnect $\Sigma$); the reverse operation is called {\it stabilization}. Note conversely that $(\Sigma,L)$ is {\it nonstabilized} if and only if every diagram $D$ of $L$ on $\Sigma$ is {\it cellular}, meaning that $D$ cuts $\Sigma$ into disks.

\begin{convention}\label{Conv:EquivD}
We regard two pairs $(\Sigma,L)$ and $(\Sigma',L')$ as equivalent if there is a pairwise homeomorphism $h:(\Sigma\times I,L)\to(\Sigma'\times I,L')$ under which $\Sigma_+\to\Sigma'_+$, respecting orientations. We regard two pairs $(\Sigma,D)$ and $(\Sigma',D')$ as equivalent if there is a pairwise homeomorphism $(\Sigma,D)\to(\Sigma',D')$  in which $\Sigma\to \Sigma'$ respects orientations  and  $D\to D'$ respects crossing information.  
\end{convention}

\begin{theorem}[Theorem 1 of \cite{kup03}]\label{T:Kup}
The stable equivalence class of any pair $(\Sigma,L)$ contains a unique nonstabilized representative.
\end{theorem}

We call a pair $(\Sigma,L)$ {\bf split} if $L$ has a disconnected diagram on $\Sigma$.  In particular, $(\Sigma,L)$ is split whenever $\Sigma$ is disconnected.  Equivalently, $(\Sigma,L)$ is split if, for some (possibly empty) disjoint union of circles $\gamma\subset \Sigma$, $\Sigma\setminus \gamma$ is disconnected, and $L$ can be isotoped so that it is disjoint from $\gamma\times I$  but intersects each component of $(\Sigma\times I)\setminus(\gamma\times I)$.

Kuperberg's theorem implies that when $(\Sigma,L)$ is nonsplit, $(\Sigma,L)$ is nonstabilized if and only if $\Sigma$ has {\it minimal genus} in its stable equivalence class.  
Note that when $\Sigma$ is connected, if $(\Sigma,L)$ is split, then it is also stabilized. The converse is false. In fact, by Kuperberg's theorem, the number of split components is an invariant of stable equivalence classes.

If $L$ is nonsplit and $g(\Sigma)>0$, then $(\Sigma\times I)\setminus L$ is irreducible, as $\Sigma\times I$ is always irreducible, since its universal cover is $\R^2\times \R$ \cite{csw14}.
The converse of this, too, is false. Indeed, if $(\Sigma_i\times I,L_i)$ is nonsplit and $g(\Sigma_i)>0$ for $i=1,2$, then construct $(\Sigma_1\#\Sigma_2,L_1\sqcup L_2)$ in the natural way.  This is split and yet, by a standard innermost circle argument, irreducible.

We will use the following result of Boden--Karimi and our subsequent generalization. 

\begin{fact}[Corollary 3.6 of \cite{bk20}]\label{F:bk36}
If $\Sigma$ is connected and $(\Sigma,L)$ has a cellular alternating diagram $D\subset\Sigma$, then $(\Sigma,L)$ is nonsplit and nonstabilized.
\end{fact}

\begin{cor}\label{C:bk36++}
Suppose $(\Sigma,L)$ is represented by an alternating diagram $D\subset\Sigma$. Then $(\Sigma,L)$ is nonsplit if and only if $D$ is connected, and $(\Sigma,L)$ is nonstabilized if and only if $D$ is cellular.
\end{cor}

\begin{proof}
If $D$ is disconnected, then $(\Sigma,L)$ is split. If $D$ is not cellular, then $(\Sigma,L)$ is stabilized.  If $D$ is cellular, then, by Fact \label{F:bk36}, each component of $(\Sigma,L)$ is nonstabilized, so $(\Sigma,L)$ is nonstabilized.  

Finally, assume that $D$ is connected and alternating.  If $D$ is cellular, then $(\Sigma,L)$ is nonsplit by Fact \label{F:bk36}. Otherwise, there exists a circle $\gamma\subset \Sigma\setminus D$ which is essential on $\Sigma$.  Destabilizing $(\Sigma,L)$ along $\gamma\times I$ preserves the fact that $D$ is connected and alternating and eventually makes $D$ cellular.  Now Fact \label{F:bk36} implies that $D$ represents a nonsplit link, and, again, the number of split components is an invariant of stable equivalence classes.
\end{proof}

\subsection{Virtual link diagrams}\label{S:Virtual}

A {\it virtual link diagram} is the image of an immersion $\bigsqcup S^1\to S^2$ in which all self-intersections are transverse double-points, some  labeled with over-under information. These labeled double-points are called {\it classical crossings}, and the other double-points are called {\it virtual crossings}. Traditionally, virtual crossings are marked with a circle, as in Figure \ref{Fi:RMoves}.  
A {\it virtual link} is an equivalence class of virtual diagrams under generalized Reidemeister moves (R-moves), of which there are seven types, three {\it classical} and four {\it non-classical}, shown up to symmetry in Figure \ref{Fi:RMoves}. 

\begin{figure}
\begin{center}
\includegraphics[width=\textwidth]{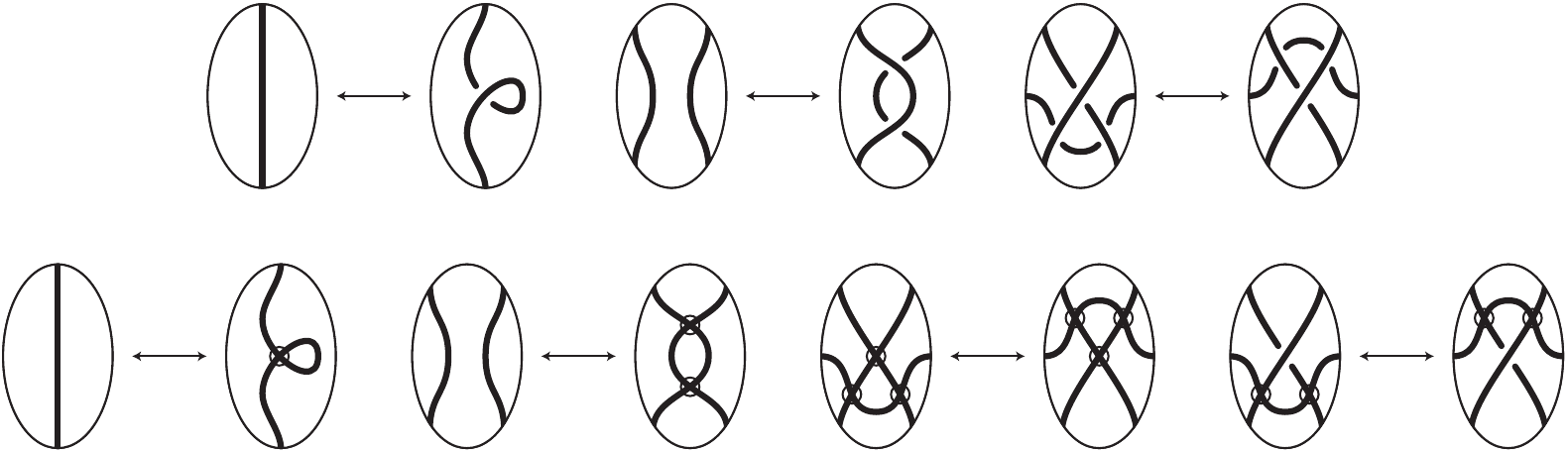}
\caption{Classical (top) and virtual (bottom) Reidemeister moves}
\label{Fi:RMoves}
\end{center}
\end{figure}

\begin{convention}\label{Conv:EquivV}
We regard two virtual link diagrams $V,V'\subset S^2$ as equivalent if they are related by planar isotopy.%
\footnote{Or equivalently if there is a pairwise homeomorphism $(S^2,V)\to(S^2,V')$  in which $S^2\to S^2$ respects orientations  and  $V\to V'$ respects crossing information.}
\end{convention}

 \begin{notation}\label{N:[V]}
Given a virtual link diagram $V$, let $[V]$ denote the set of all diagrams related to $V$ by {\it non-classical} R-moves. 
\end{notation}

We call a virtual link {\bf split} if it has a disconnected diagram, and we call a virtual link diagram $V$ {\bf split} if some $V'\in [V]$ is disconnected. We call $V$ {\bf alternating} if its classical crossings alternate between over- and under-crossings (ignore virtual crossings by passing straight through them). If a non-split virtual link $K$ corresponds to a non-stabilized pair $(\Sigma,L)$, then the genus of $\Sigma$ is called the genus of $K$, written $g(K)$.


\section{Lassos, lasso diagrams, and lasso numbers}\label{S:Lassos}

We introduce lassos for pairs $(\Sigma,D)$ and then for virtual link diagrams. In both contexts, a lasso is a disk that captures the salient features of the diagram.

\subsection{Lassos of link diagrams on closed surfaces}\label{S:lassoSigma}


\begin{definition}
A {\bf lasso} for $(\Sigma,D)$ is a disk $X\subset \Sigma$ that intersects $D$ generically and contains all crossings of $D$; $X$ is {\bf acceptable} if:
\begin{itemize}
\item $D\cap X$ is connected, and
\item $D\cut X$ is comprised of arcs, none parallel in $\Sigma\cut X$ to $\partial X$.
\end{itemize}
\end{definition}


\begin{prop}\label{P:LassoSD}
A pair $(\Sigma,D)$ admits an acceptable lasso if and only if $D$ is connected.  
\end{prop}


\begin{proof}
If $D$ is disconnected and $X$ is a lasso for $(\Sigma,D)$ such that $D\cap X$ is connected, then $D$ must have a component 
 without crossings which lies entirely outside $X$, so $X$ cannot be acceptable.
For the converse, let $G$ be the underlying graph of $D$.%
\footnote{Obtain $G$ from $D$ by replacing each crossing with an unlabeled 4-valent vertex.} 
Take a spanning tree $T$ of $G$ and consider its regular neighborhood $\nu T\subset\Sigma$. This is a lasso for $(\Sigma,D)$. If any arc of $D\cut X$ is parallel in $\Sigma\cut X$ to $\partial X$, isotope $\partial X$ past that arc; note that $D\cap X$ remains connected.  Repeat until $X$ is acceptable.
\end{proof}

Proposition \ref{P:LassoSD} and Corollary \ref{C:bk36++} imply:

\begin{cor}\label{C:LassoAltSD}
An alternating diagram $D$ on a surface $\Sigma$ is connected, and thus represents a nonsplit link, if and only if $(\Sigma,D)$ admits an acceptable lasso.
\end{cor}

Note that $(\Sigma,D)$ cannot admit any lasso whatsoever if $\Sigma$ is disconnected.  This motivates:

\begin{definition}
Let $D_1,\hdots, D_n$ be the components of a link diagram $D\subset\Sigma$, and for each $i$ let $\Sigma_i$ be the component of $\Sigma$ that contains $D_i$.  A {\bf lasso system} for $(\Sigma,D)$ is a collection of disks $X_1,\hdots,X_n$ such that each $X_i$ is a lasso for $(\Sigma_i,D_i)$; $X$ is {\bf acceptable} if each $X_i$ is acceptable.
\end{definition}

Proposition \ref{P:LassoSD} and Corollary \ref{C:LassoAltSD} imply:

\begin{cor}\label{C:System}
Every pair $(\Sigma,D)$ admits an acceptable lasso system $X$, and if $D$ is alternating then $|X|=|D|$.
\end{cor}

\subsection{Lassos for virtual link diagrams}\label{S:lassoV}

\begin{definition}
A {\bf lasso} for a virtual link diagram $V\subset S^2$ is a disk $X\subset S^2$ that intersects $V$ generically and contains all classical crossings of $V$ but no virtual crossings. Call $V\cap X$ and $V\cut X$ respectively the {\bf classical} and {\bf virtual} parts of $V$ (with respect to $X$).  Say that $X$ is {\bf acceptable} if:
\begin{itemize}
\item the classical part of $V$ is connected and
\item the virtual part of $V$ is comprised of arcs, none of which:
\begin{itemize}
\item is parallel in $S^2\cut(V\cup X)$ to $\partial X$,
\item self-intersects, nor 
\item intersects any other arc more than once.  
\end{itemize}
\end{itemize}
\end{definition}

Whereas not every pair $(\Sigma,D)$ admits a lasso, every virtual diagram does:

\begin{prop}\label{P:VLasso}
Every virtual link diagram $V$ admits a lasso.
\end{prop}

\begin{proof}
Construct a planar graph $G\subset S^2$ as follows.  Place one vertex at each classical crossing of $V$ and one in the interior of each component of $S^2\cut V$.  Each classical crossing $c$ of $V$ lies on the boundary of three or four components of $S^2\cut V$; construct an edge from $c$ to the vertex in each of these components.  Each edge $e$ of $V$ is incident to two components of $S^2\cut V$; construct an edge $\alpha$ between the vertices in these components, such that $\alpha\cap V$ consists of a single point in the interior of $e$.  Now choose a spanning tree $T$ for $G$, and take a regular neighborhood $\nu T$ of $T$ in $S^2$.  The disk $\nu T\subset S^2$ intersects $V$ generically and contains all classical crossings in $V$ but no virtual crossings. Thus, $\nu T$ is a lasso for $V$ (not necessarily acceptable).  
\end{proof}

Further:

\begin{prop}\label{P:VLassoAcc}
Given a nonsplit virtual link diagram $V$, some $V^*\in [V]$ admits an acceptable lasso.
\end{prop}

\begin{proof}
By Proposition \ref{P:VLasso}, $V$ has a lasso, $X$. If $X\cap V$ is disconnected, then there is a properly embedded arc $\alpha\subset X$ disjoint from $V$ such that $V$ intersects both disks of $X\cut\alpha$.  Take a regular neighborhood $\nu\alpha\subset X$ disjoint from $V$. Some arc $\beta$ in the virtual part of $V$ has one endpoint on each disk of $X\cut\nu\alpha$, or else $V$ would be split.  While fixing $V\cap X$ and $\beta$, change $V\to V'$ by non-classical R-moves so that all virtual crossings of $V'$ lie outside the disk $X'=(X\cut\nu\alpha)\cup\nu\beta$. (This may require making $V'\cap\nu\alpha\neq\varnothing$.) Now $X'$ is a lasso for $V'$, and $|X'\cap V'|=|X\cap V|-1$.  Repeating this process eventually yields a lasso $X''$ for some $V''\in[V]$ such that $X''\cap V''$ is connected and $V''\cut X''$ is comprised of arcs. 

If any arc of $V''\cut X''$ is parallel in $S^2\cut (V''\cup X'')$ to $\partial X''$, isotope $\partial X''$ past that arc (note that $V''\cap X''$ remains connected); if some arc of $V''\cut X''$ self-intersects, or if two such arcs intersect more than once, perform non-classical R-moves to remove such intersections.  Repeat. Eventually, this yields an acceptable lasso $X^*$ for some $V^*\in[V]$.
\end{proof}

It is again natural to define:

\begin{definition}
Let $V_1,\hdots, V_n$ be the components of a virtual link diagram $V$.  A {\bf lasso system} for $V$ is a collection of disks $X_1,\hdots,X_n$ such that each $X_i$ is a lasso for $V_i$; $X$ is {\bf acceptable} if each $X_i$ is acceptable.
\end{definition}

Proposition \ref{P:VLassoAcc} implies:

\begin{cor}\label{C:VSystem}
Given any virtual link diagram $V$ some $V^*\in [V]$ admits an acceptable lasso system.
\end{cor}

\subsection{Lasso diagrams and lasso numbers}\label{S:LD}
 
 \begin{figure}
 \begin{center}
 \labellist\tiny
 \pinlabel {1} at 0 80
 \pinlabel {2} at 0 220
 \pinlabel {2} at 230 80
 \pinlabel {1} at 230 220
 \pinlabel {1} at 310 40
 \pinlabel {3} at 310 280
 \pinlabel {3} at 405 -15
 \pinlabel {4} at 500 40
 \pinlabel {2} at 500 280
 \pinlabel {1} at 405 320
 \pinlabel {4} at 527 180
 \pinlabel {2} at 527 124
  \pinlabel {2} at 580 80
 \pinlabel {1} at 805 80
  \pinlabel {1} at 580 220
 \pinlabel {2} at 805 220
  \pinlabel {1} at 890 40
 \pinlabel {3} at 890 280
 \pinlabel {3} at 985 -15
 \pinlabel {4} at 1080 40
 \pinlabel {2} at 1080 280
 \pinlabel {1} at 985 320
  \pinlabel {4} at 1107 180
 \pinlabel {2} at 1107 124
  \pinlabel {1} at 1180 40
 \pinlabel {3} at 1180 280
 \pinlabel {3} at 1265 -15
 \pinlabel {4} at 1365 40
 \pinlabel {2} at 1360 280
 \pinlabel {1} at 1265 320
  \pinlabel {4} at 1397 180
 \pinlabel {2} at 1397 124
  \pinlabel {1} at 1440 110
 \pinlabel {2} at 1435 170
 \pinlabel {3} at 1510 0
 \pinlabel {2} at 1670 70
 \pinlabel {1} at 1685 180
 \pinlabel {3} at 1565 320
  \pinlabel {1} at 1730 110
 \pinlabel {2} at 1725 170
 \pinlabel {3} at 1800 0
 \pinlabel {2} at 1960 70
 \pinlabel {1} at 1975 180
 \pinlabel {3} at 1855 320   \endlabellist
 \includegraphics[width=\textwidth]{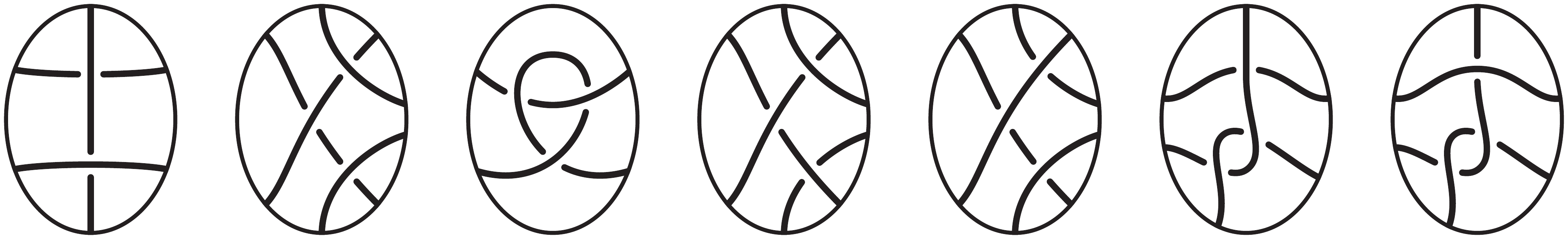}
 \caption{Lasso diagrams of the first seven non-classical virtual knots}
 \label{Fi:Table}
 \end{center}
 \end{figure}
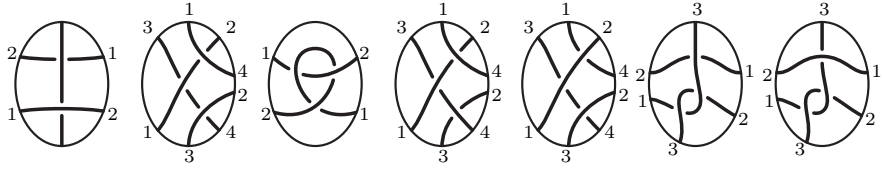
 
 Suppose $X$ is a lasso for a nonsplit virtual link diagram $V$.  Take the classical part of $V$ and label its endpoints in pairs following the virtual part of $V$.  We call $(X,V\cap X)$, with labels on $\partial X$, a {\bf lasso diagram} from $V$.  A lasso $Y$ for a pair $(\Sigma,D)$ likewise gives a lasso diagram $(Y,D\cap Y)$. Figure \ref{Fi:Table} shows lasso diagrams of the first seven non-classical virtual knots \cite{green}.
 
 We regard two lasso diagrams $(X,V)$ and $(X',V')$ as equivalent if there is a pairwise homeomorphism $(X,V)\to (X',V')$ such that $
 V\to V'$, $X\to X'$, and $\partial X\to \partial X'$ respectively preserve crossing information, orientations, and pairings between labels.
 
 Given a lasso diagram $(X,V\cap X)$, it is easy to construct a virtual link diagram $V'$ for which $X$ is a lasso, hence $V'\cap X=V\cap X$, and thus $[V']=[V]$, although the virtual parts of $V$ and $V'$ need not match. Call $(X,V\cap X)$ {\bf acceptable} if $V\cap X$ is connected and no label appears twice consecutively on $\partial X$.  
 
 \begin{obs}
  Suppose $X$ is a lasso for a nonsplit virtual link diagram $V$. The lasso diagram $(X,V\cap X)$ is acceptable if and only if $X$ is an acceptable lasso for some $V'\in[V]$ with $V'\cap X=V\cap X$.
 \end{obs}
 
 \begin{definition}\label{D:LN}
Given a diagram $D\subset \Sigma$ of a link $L\subset\Sigma\times I$, denote the set of all lassos for $(\Sigma,D)$ by $\text{lassos}(\Sigma,D)$.  Define the {\bf lasso number} of $(\Sigma,D)$ to be
\[\text{lasso}(\Sigma,D)=\min_{X\in\text{lassos}(\Sigma,D)}|\partial X\cap D|.\]
Define the {\bf lasso number} of $(\Sigma,L)$ to be
\[\text{lasso}(\Sigma,L)=\min_{\text{diagrams }D\text{ of }L}\text{lasso}(\Sigma,D).\]

Say that a lasso $X$ for a diagram $(\Sigma,D)$ of $L\subset\Sigma\times I$ is {\it efficient for $D$}, or $D$-\emph{efficient}, if $|\partial X\cap D|=\text{lasso}(\Sigma,D)$. Say that $X$  is {\it efficient for $L$}, or $L$-\emph{efficient}, if $|\partial X\cap D|=\text{lasso}(\Sigma,L)$. 
\end{definition}

Note that, for any diagram $D$ of any nonsplit pair $(\Sigma,L)$, every $D$-efficient lasso is acceptable.  In general, computing lasso numbers of diagrams (let alone links!) may appear difficult. In practice, however, computing lasso numbers of diagrams is sometimes straightforward, due to the following fact:

\begin{prop}
Consider a pair $(\Sigma,D)$, where $\Sigma$ has genus $g$.  Then every acceptable lasso $X$ for $(\Sigma,D)$ satisfies
\[\frac{1}{2}|\partial X\cap D|=2g-1+|\Sigma\cut D|-n,\]
where $n$ denotes the number of faces of $\Sigma\cut D$ contained entirely in $X$.
Hence, $X$ is efficient for $D$ if and only if it contains as many faces of $\Sigma\cut D$ as possible.
\end{prop}

\begin{proof}
Collapse any acceptable lasso $X$ to a point to obtain a graph $G=D/X$ in the surface $\Sigma/X\equiv \Sigma$.  Then $G$ has one vertex and $\frac{1}{2}|\partial X\cap D|$ edges. Let $n$ denote the number of faces of $\Sigma \cut D$ contained entirely in $X$.

The fact that $X$ is acceptable implies that $D\cap X$ is connected, and thus that for each face $U$ of $\Sigma\cut D$, $U\cut X$ is either empty or connected.    Therefore, $|\Sigma\cut G|=|\Sigma\cut D|-n$.  Hence, as claimed:
\begin{align*}
2-2g=\chi(\Sigma)&=1-\frac{1}{2}|\partial X\cap D|+|\Sigma\cut D|-n\\
\frac{1}{2}|\partial X\cap D|&=2g-1+|\Sigma\cut D|-n.
\end{align*}
Of the quantities on the right-hand side, only $n$ depends on $X$.  Thus, $X$ is efficient for $D$ if and only if it contains as many faces of $\Sigma\cut D$ as possible.
\end{proof}

Denoting the set of faces of $\Sigma\cut D$ by $\{U_i\}_{i\in\mathcal{C}}$, the problem of computing the lasso number of $(\Sigma,D)$ is thus equivalent to finding the largest subset $\mathcal{L}\subset\mathcal{C}$ such that $\bigcup_{i\in\mathcal{L}}U_i$ is simply connected.

\begin{prop}
If $(\Sigma,D)$ is cellular and connected, then:
\begin{itemize}
\item every $D$-efficient lasso $X\subset \Sigma$ contains every monogon of $\Sigma\cut D$, and
\item there is a $D$-efficient lasso $X\subset \Sigma$ that contains every bigon of $\Sigma\cut D$.
\end{itemize}
\end{prop}

\begin{figure}
\begin{center}
\includegraphics[width=\textwidth]{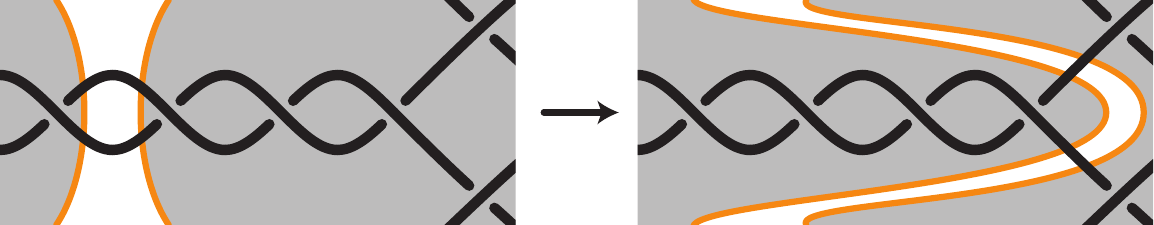}
\caption{Capturing bigons in a $D$-efficient lasso (shaded)}
\label{Fi:Capture}
\end{center}
\end{figure}

\begin{proof}
Assume $\Sigma$ has positive genus, or else the claim is trivial, and consider a $D$-efficient lasso $X$.  It is easy to see that it must contain every monogon of $\Sigma\cut D$. For each bigon $Y$ of $\Sigma\cut D$,  $X$ either contains $Y$ or can be modified near $Y$ as shown in Figure \ref{Fi:Capture} to increase the number of bigons $X$ contains, while preserving $|\partial X\cap D|$.  After enough such modifications, $X$ will still be $D$-efficient and
will contain every bigon of $\Sigma\cut D$.
 \end{proof}

Computing lasso numbers of large diagrams without bigons or monogons, however, seems challenging.  
We pose two problems:

\begin{problem}
Describe an efficient algorithm that computes the lasso number of an arbitrary connected diagram $D\subset \Sigma$.
\end{problem}

\begin{question}
Do we have $\text{lasso}(\Sigma,D)=\text{lasso}(\Sigma,L)$ for every cellular alternating link diagram $D\subset\Sigma$ of a link $L\subset\Sigma\times I$?
  \end{question}
 
\section{Correspondences}\label{S:Corr}

\subsection{Main correspondence}\label{S:MainCorr}

\begin{figure}
\begin{center}
\includegraphics[width=.4\textwidth]{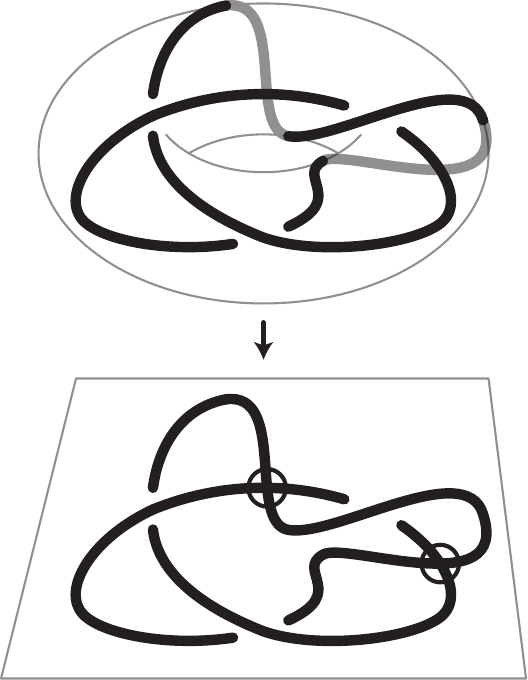}
\caption{A link diagram on the torus and a corresponding virtual diagram}
\label{Fi:Correspondence}
\end{center}
\end{figure}

We will establish Correspondence \ref{C:Diagrams} by showing that the following two constructions are inverses.  For clarity, we describe both constructions in the nonsplit setting; adapting to split links is straightforward. (This is why we defined lasso {\it systems} in \textsection\ref{S:Lassos}.)

\begin{construction}\label{C:LassoVtoD}
Given a nonsplit virtual link diagram $V'$, choose some $V\in [V']$ which admits an acceptable lasso $X$. Construct a pair $(\Sigma,D)$ by taking a regular neighborhood of $X\cup V$, changing the neighborhood of each virtual crossing as in Figure \ref{Fi:VtoA}, and (abstractly) capping off each resulting boundary component with a disk. Then $X$ embeds naturally in $\Sigma$ and is an acceptable lasso for $D$.
\end{construction}


\begin{construction}\label{C:LassoDtoV}
Suppose that $X$ is an acceptable lasso for $(\Sigma,D)$, which is cellular (and connected).
\begin{itemize}
\item Choose an embedding $\varphi:\Sigma\to\wh{S^3}$ 
such that $\varphi(X)$ lies entirely on the front of $\Sigma$ and $(\pi\circ\varphi(D\cap X))\cap(\pi\circ\varphi(D\setminus X))=\varnothing$.\footnote{A regular point $x$ of $\pi|_{\varphi(\Sigma)}$ lies on the {\it front} or {\it back} of $\varphi(\Sigma)$ depending on whether an {\it even} or {\it odd} number of points of $\varphi(\Sigma)$ lie directly above it, or equivalently whether $\pi\circ\varphi:\Sigma\to S^2$ is orientation-preserving or -reversing near $\varphi^{-1}(x)$.}
\item Perturb $\varphi$ so that all self-intersections in $\pi\circ\varphi(D\cut X)$ are transverse double-points. 
\item Then $\pi\circ\varphi(X)$ is an acceptable lasso for the virtual link diagram $\pi\circ\varphi(D)=V'\in[V]$.
\end{itemize} 
\end{construction}

\begin{correspondence}\label{C:Diagrams}
Constructions \ref{C:LassoVtoD} and \ref{C:LassoDtoV} are inverses and thus define a correspondence between cellular pairs $(\Sigma,D)$ and equivalence classes $[V]$ (both with specified acceptable lassos).
\end{correspondence}

\begin{proof}
It suffices to prove the nonsplit case. Extending to the split case is straightforward.

If $V'$ is a nonsplit virtual link diagram and $V\in[V']$ admits an acceptable lasso $X$, then applying Construction \ref{C:LassoVtoD} to $V$ and $X$ and then applying Construction \ref{C:LassoDtoV} to the result gives a pair $V''$ and $X''$ such that the lasso diagrams $(X,V\cap X)$ and $(X'',V''\cap X'')$ are equivalent.  Therefore, $V''\in [V]=[V']$.

If $X$ is an acceptable lasso for $(\Sigma,D)$, which is cellular and connected, then applying Construction \ref{C:LassoDtoV} to $(\Sigma,D)$ and $X$ and then applying Construction \ref{C:LassoVtoD} to the result gives a pair $(\Sigma',D')$ and a lasso $X'$ such that the lasso diagrams $(X,D\cap X)$ and $(X',D'\cap X')$ are equivalent.  Therefore, $(\Sigma,D)$ and $(\Sigma',D')$ are equivalent.  

Thus, Constructions \ref{C:LassoVtoD} and \ref{C:LassoDtoV} are inverses.  Moreover, given some $[V]$, Construction \ref{C:LassoVtoD} yields the same pair $(\Sigma,D)$ as the construction described in the theorem as state in \textsection\ref{S:intro}.  Therefore, that construction determines the same correspondence as Constructions \ref{C:LassoVtoD} and \ref{C:LassoDtoV}.
 \end{proof}

Correspondence \ref{C:Diagrams} gives a new diagrammatic perspective on a well-known correspondence:

\begin{correspondence}[\cite{kauff98,kaka,cks02}]\label{C:Links}
There is a correspondence between virtual links and stable equivalence classes of links in thickened surfaces.  Namely, choose any representative diagram and apply the diagrammatic Correspondence \ref{C:Diagrams}.
\end{correspondence}

Correspondences \ref{C:Diagrams} and \ref{C:Links} allow us to translate Corollary  \ref{C:LassoAltSD} to virtual link diagrams:

\begin{cor}\label{C:VSplit}
An alternating virtual link diagram $V$ represents a nonsplit virtual link $K$ if and only if some $V'\in[V]$ admits an {\it acceptable} lasso.
\end{cor}

\begin{proof}
If $K$ is nonsplit, then every $V'\in[V]$ is connected, so Proposition \ref{P:VLassoAcc} implies that some $V'\in[V]$ admits an acceptable lasso.  Conversely, if some $V'\in[V]$ admits an acceptable lasso, then Construction \ref{C:LassoVtoD} gives an acceptable lasso for the associated cellular alternating pair $(\Sigma,D)$. Corollary \ref{C:LassoAltSD} now implies
that $(\Sigma,D)$ represents a nonsplit link $L\subset\Sigma\times I$. Finally, since every diagram of $L$ on $\Sigma$ is connected, Correspondences \ref{C:Diagrams} and \ref{C:Links} imply that every virtual diagram of $K$ is connected. That is, $K$ is nonsplit.
\end{proof}

\subsection{Key subtlety}\label{S:Subtle}

There is an important caveat in Correspondences \ref{C:Diagrams} and \ref{C:Links} which is worth noting explicitly. It arose in Construction \ref{C:LassoDtoV}, where we chose $\varphi$ such that all crossings of $\varphi(D)$ lie on the front of $\varphi(\Sigma)$.

\begin{rem}\label{R:front}
In order for Construction \ref{C:LassoDtoV} to respect the Correspondences \ref{C:Diagrams} and \ref{C:Links}, we need the requirement that all crossings of $\varphi(D)$ lie on the front of $\varphi(\Sigma)$; otherwise, different embeddings $\Sigma\to \wh{S^3}$ may yield distinct virtual links.  See Example \ref{Ex:front}.
\end{rem}

\begin{figure}
\begin{center}
\labellist
\hair4pt
\pinlabel {\scalebox{.9}{$D$}} [c] at 272 57
\pinlabel {\scalebox{.9}{$D'$}} [c]  at 488 57
\pinlabel {\scalebox{.9}{$D''$}} [c]  at 68 57
\endlabellist
\includegraphics[width=.8\textwidth]{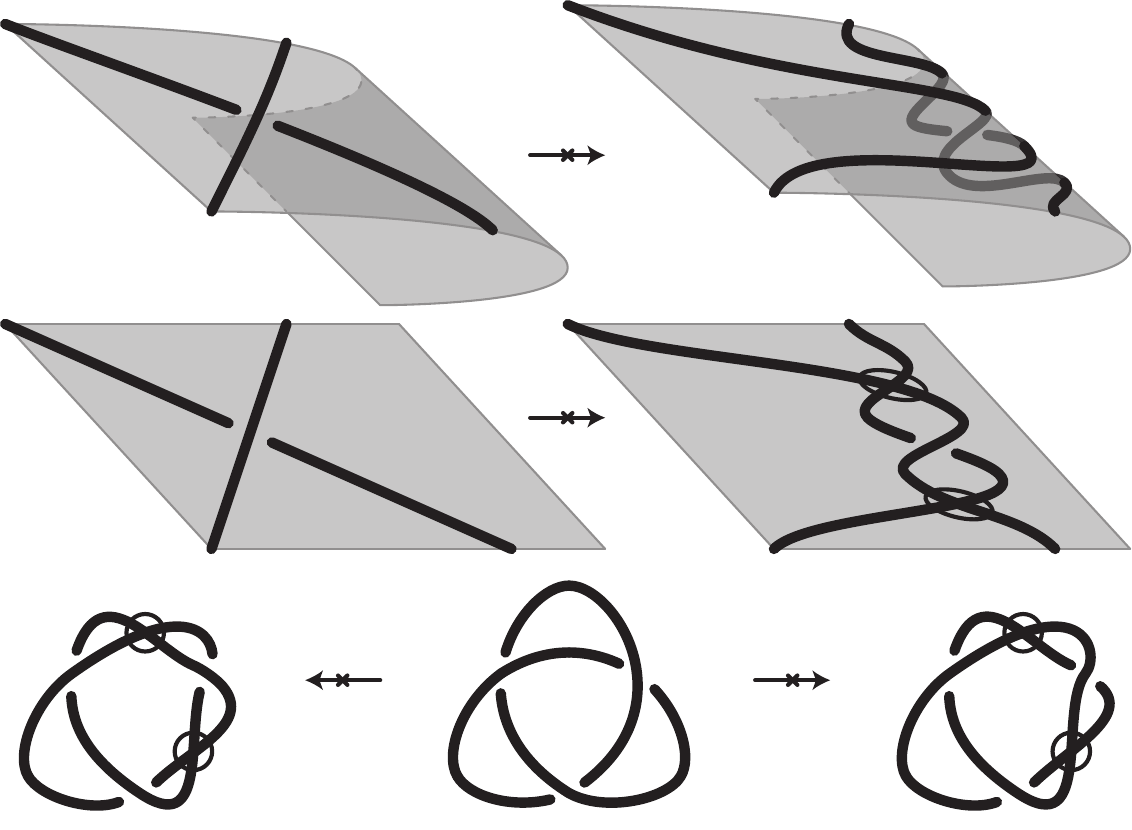}
\caption{Given $D\subset\Sigma$, one obtains $V\subset S^2$ by embedding $\Sigma$ in $S^3$ and projecting, but all crossings of $D$ must remain on the front of $\Sigma$.}
\label{Fi:PastCritical}
\end{center}
\end{figure}

\begin{example}\label{Ex:front}
Let $D$ be a minimal diagram of the RH trefoil on a 2-sphere $\Sigma$, and embed $\Sigma$ in $\wh{S^3}$ such that the critical locus of $\pi|_\Sigma$ is a simple closed curve and $D$ lies entirely on the front of $\Sigma$.  The corresponding virtual diagram $V$ is also a minimal diagram of the classical RH trefoil.  Now isotope $D$ on $\Sigma$ as shown in Figure \ref{Fi:PastCritical}, so that a crossing passes across the critical circle of $\pi|_\Sigma$.  Denote the resulting diagram on $\Sigma$ by $D'$. The virtual diagram $V'=\pi\circ\phi(D')$ represents the virtual knot $3.5$, which is distinct from the classical RH trefoil \cite{kauff98}, even though $D$ and $D'$ are isotopic on $\Sigma$. 

Interestingly, the virtual knot 3.5 has the same Jones polynomial as the RH trefoil, but the two can be distinguished using the involutory quandle, also called the fundamental quandle. Indeed, by Lemma 5 of \cite{kauff98}, the virtual knot 3.5 has the same involutory quandle as the unknot, which is distinct from that of the RH trefoil, since the former is trivial and the latter is not \cite{joyce}.

We note too that the diagram $D''$ in Figure \ref{Fi:PastCritical} represents the virtual knot $3.7$, which has trivial Jones polynomial and thus is also distinct from the RH trefoil.
\end{example}

\subsection{Diagrammatic extension of Kuperberg's Theorem}
%
%
Given a virtual diagram $V$ of, say, a nonsplit virtual link, one may define the {\it genus} $g(V)$ to be the genus $g(\Sigma)$, where $(\Sigma,D)$ corresponds to $[V]$ under Correspondence \ref{C:Diagrams}.  Note that $g([V])$ is well-defined. Say that $V$ has {\it minimal genus} if $g(V)\leq g(V')$ for all diagrams $V'$ of the same virtual link.  In general, given two diagrams $V$ and $V'$ of the same virtual link, it is plausible that {\it all} sequences $V=V_0\to\cdots\to V_n=V'$ might have some $V_i$ with $g(V_i)\gg\max\{g(V),g(V')\}$.  When $V$ and $V'$ have minimal genus, though, we have the following diagrammatic extension of Kuperberg's theorem:

\begin{theorem}\label{T:KupD}
All minimal genus diagrams of a nonsplit virtual link are related by minimal-genus-preserving generalized R-moves.
\end{theorem}

\begin{proof}
Let $V$ and $V'$ be two such virtual diagrams, let $(\Sigma,D)$ and $(\Sigma',D')$ be the  diagrams corresponding to $[V]$, $[V']$ under Correspondence \ref{C:Diagrams}, and let $L\subset\Sigma\times I$ and $L'\subset\Sigma'\times I$ be the links they represent.  The minimal genus condition ensures that $L$ and $L'$ are nonstabilized, so by Kuperberg's Theorem there is a pairwise orientation-preserving homeomorphism $(\Sigma\times I,L)\to(\Sigma'\times I,L')$. Therefore, there is a sequence $D=D_0\to\cdots\to D_n$ of R-Moves  on $\Sigma$ such that there is a pairwise homeomorphism $(\Sigma,D_n)\to(\Sigma',D')$.  Applying Correspondence \ref{C:Diagrams} to the sequence $(\Sigma,D_0)\to\cdots\to(\Sigma,D_n)$ gives a sequence $[V]=[V_0]\to\cdots\to[V_n]=[V']$ where each $g([V_i])$ is minimal and each pair $[V_i],[V_{i+1}]$ have representatives that differ by a single classical R-move; refining this sequence with non-classical R-moves gives the desired sequence of generalized R-moves taking $V$ to $V'$.
\end{proof}

\subsection{Extending the correspondences to lasso diagrams}

Consider the following moves on lasso diagrams, and see Figure \ref{Fi:Moves}):

\begin{move}\label{M:LDMove1}
Change a lasso diagram $(X,V)$ to another $(X,V')$ which we construct as follows (or do the inverse operation):
\begin{itemize}
\item Choose a properly embedded arc $\alpha\subset X\cut V$ with one endpoint $x\in \partial X\setminus V$ and the other $v\in V\cap\text{int}(X)$, not at a crossing.
\item Fixing $V\cut \nu v$, isotope $V\cap\nu v$ through $\nu \alpha$ to $\partial X$, creating two new endpoints, and call this $V'$; 
\item Keep all the labels on $\partial X$, and add a new pair of labels for the two new endpoints.
\end{itemize}
\end{move}

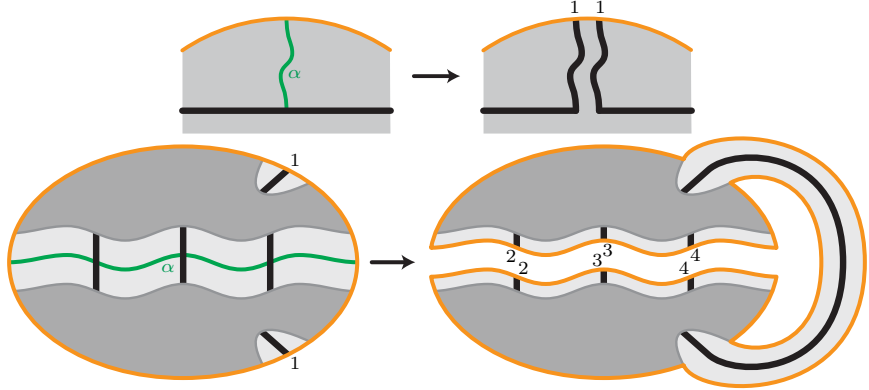
\begin{figure}
\begin{center}
\labellist \tiny
\pinlabel {$\FG{\alpha}$} at 90 47
\pinlabel {1} at 308 99
\pinlabel {1} at 328 99
\endlabellist
\includegraphics[width=.6\textwidth]{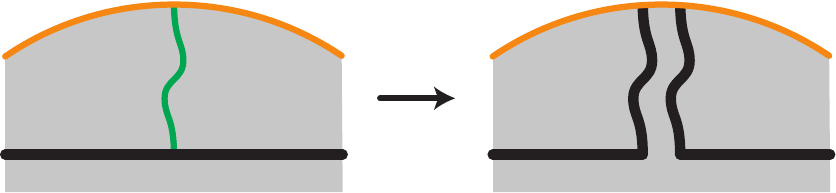}\\
\labellist \tiny
\pinlabel {$\FG{\alpha}$} at 200 150
\pinlabel {1} at 358 284
\pinlabel {1} at 358 31
\pinlabel {2} at 627 165
\pinlabel {2} at 642 149
\pinlabel {3} at 734 159
\pinlabel {3} at 747 172
\pinlabel {4} at 842 149
\pinlabel {4} at 856 166
\endlabellist
\includegraphics[width=\textwidth]{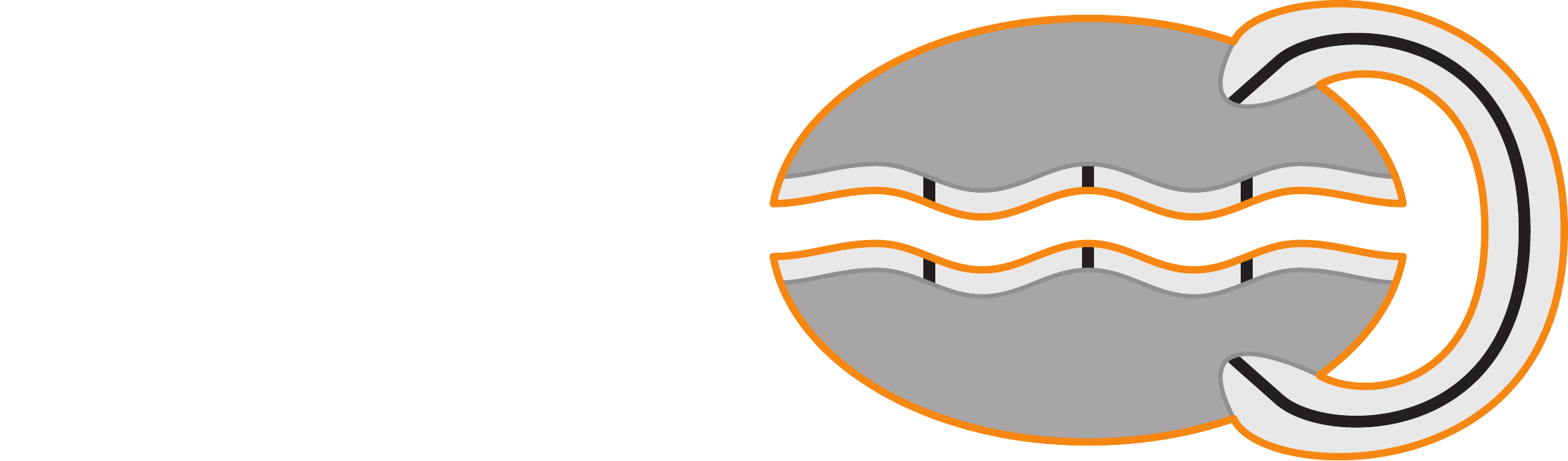}
\caption{Moves \ref{M:LDMove1} (top) and \ref{M:LDMove2} (bottom)}
\label{Fi:Moves}
\end{center}
\end{figure}

Note that a lasso diagram resulting from the type of Move \ref{M:LDMove1} described is never acceptable. Viewing $(X,V)$ as the part of a cellular pair $(\Sigma,D)$ in a lasso $X$, the inverse version of Move \ref{M:LDMove1} corresponds to isotoping an (outermost) arc of $D\cut X$ through $\Sigma\cut X$ past $\partial X$.

\begin{move}\label{M:LDMove2}
Change a lasso diagram $(X,V)$ to another $(X',V')$ which we construct as follows  (or do the inverse operation):
\begin{itemize}
\item Choose a properly embedded arc $\alpha\subset X$ that intersects $V$ generically such that some label $\ell$ appears once on each arc of $\partial X\setminus\partial\alpha$.
\item Attach an (oriented) band to $X\cut\nu\alpha$  near the two $\ell$-labels; call this $X'$;
\item Glue the core of the band to $V\cut\nu\alpha$; call this $V'$;
\item Label $\partial X'$ with all labels from $\partial X$, except the $\ell$'s, together a new pair of labels for each arc of $V\cap\nu\alpha$.
\end{itemize}
\end{move}

The diagrammatic  Correspondence \ref{C:Diagrams} extends to equivalence classes of lasso diagrams under Moves \ref{M:LDMove1}-\ref{M:LDMove2}:

\begin{theorem}\label{T:LDMoves}
Let $X_1$ and $X_2$ be lassos for nonsplit virtual link diagrams $V_1$ and $V_2$.  Then $[V_1]=[V_2]$ if and only if the lasso diagrams $(X_i,V_i\cap X_i)$ are related by a sequence of Moves \ref{M:LDMove1}-\ref{M:LDMove2}. 
\end{theorem}

\begin{proof}
Let $(\Sigma_i,D_i)$ be the cellular diagrams corresponding to $[V_i]$, let $Y_i$ be lassos for $(\Sigma_i,D_i)$ that correspond to the lassos $X_i$ for $V_i$, so that each $(X_i,V_i\cap X_i)\equiv(Y_i,D_i\cap Y_i)$, and let $G_i$ be the underlying graph of $D_i$.

Suppose that $[V_1]=[V_2]$. Then by Correspondence \ref{C:Diagrams}, we may identify $(\Sigma_1,D_1)\equiv(\Sigma_2,D_2)$, denoting this $(\Sigma,D)$, and we may write $G$ for the graph $G_1\equiv G_2$. Perform a sequence of Move \ref{M:LDMove2}'s on each $(Y_i,D\cap Y_i)$ so that in the resulting lasso diagram $(Y'_i,D\cap Y'_i)$, $D\cap Y'_i$ is connected (this is possible because $D$ is connected; see the proof of Proposition \ref{P:LassoSD} for details). Next, perform a sequence of Move \ref{M:LDMove1}'s on each $(Y'_i,D\cap Y'_i)$ so that so that in the resulting lasso diagram $(Y''_i,D\cap Y''_i)$ no disk of $\Sigma\cut D$ lies entirely in $Y''_i$.  It follows that $G$ has spanning trees $T_i$, $i=1,2$, such that each $(Y''_i,D\cap Y''_i)\equiv(\nu T_i,D\cap\nu T_i)$. Further,  $G$ has a sequence of spanning trees $T_1=T_{11}\to T_{12}\to\cdots\to T_{1n}=T_2$ such that each $T_{1i}\to T_{1(i+1)}$ replaces one edge of $T_{1i}$ with another edge.  Each corresponding move $(\nu T_i,D\cap \nu T_i)\to (\nu T_{i+1},D\cap \nu T_{i+1})$ is a Move \ref{M:LDMove2}.  

The converse is evident.
\end{proof}

It follows that the older  Correspondence \ref{C:Links} extends to equivalence classes of lasso diagrams under Moves \ref{M:LDMove1}-\ref{M:LDMove2} and classical R-moves:

\begin{cor}\label{C:Links2}
There is a triple bijective correspondence between (i) virtual links, (ii) stable equivalence classes of links in thickened surfaces, and (iii) equivalence classes of lasso diagrams under Moves \ref{M:LDMove1}-\ref{M:LDMove2} and classical Reidemeister moves.  
\end{cor}


\section{Characterizations of local and pairwise primeness}

\subsection{Preliminaries}

\subsubsection{Checkerboard colorability}\label{S:CB}

A cellular pair $(\Sigma,D)$ is {\bf checkerboard colorable} if one can color the disks of $\Sigma\cut D$ black and white so that regions of the same shade abut only at crossings. Call a virtual link diagram $V$ checkerboard colorable if the cellular pair associated to $[V]$ under Correspondence \ref{C:Diagrams} is checkerboard colorable.  
Call a virtual link $K$ checkerboard colorable if its diagrams are {\it }all checkerboard colorable.

\begin{prop}[\cite{oz06,bk20}]\label{P:CB}
Suppose $D\subset \Sigma$ is a cellular alternating diagram of a link $L\subset \Sigma\times I$. Then $(\Sigma,D)$ is checkerboard colorable, and $L$ is nullhomologous over $\Z/2$.
\end{prop}

\begin{proof}
As $D$ is cellular alternating, one may orient each disk of $\Sigma\cut D$
so that, under the resulting boundary orientation, over-strands run toward crossings and under-strands away from crossings.  Since $\Sigma$ is orientable, one may compare these local orientations in a globally coherent way. This partitions the disks of $\Sigma\cut D$ into two classes, determining a checkerboard coloring from which one can construct {\it checkerboard surfaces} 
 for $L$, 
so $L$ is nullhomologous over $\Z/2$.
\end{proof}

\begin{rem}
By similar reasoning, cellular alternating link diagrams on nonorientable surfaces are {\it never} checkerboard colorable. 
\end{rem}

\begin{rem}\label{R:CBStab}
If $L\subset\Sigma\times I$ is nullhomologous over $\Z/2$, then this remains true under stabilization, but not necessarily destabilization.
\end{rem}

Remark \ref{R:CBStab} and Proposition \ref{P:CB} imply:

\begin{cor}
A virtual link $K$ is checkerboard colorable if and only if, for the minimal genus pair $(\Sigma,L)$ associated to $K$ under Correspondence \ref{C:Links} and Kuperberg's theorem, $L$ is nullhomologous over $\Z/2$.
\end{cor}

\subsubsection{Pairwise connect sums: local and annular}\label{S:Sums}
Given two 
manifolds $M_1$ and $M_2$ of dimension $n$, a connect sum $M_1\#M_2$ is constructed by removing the interior of an $n$-ball $B_i$ from each $M_i$ and gluing $M_1\setminus\text{int}(B_1)$ to $M_2\setminus\text{int}(B_2)$ by a 
homeomorphism $\partial B_1\to\partial B_2$. If $M_1$ and $M_2$ are oriented, one requires this homeomorphism to be orientation-reversing.

Given two pairs $(\Sigma_1,D_1)$ and $(\Sigma_2,D_2)$, a {\bf pairwise connect sum} $(\Sigma_1,D_1)\#(\Sigma_2,D_2)=(\Sigma,D)$ is constructed by taking $\Sigma_1\#\Sigma_2=\Sigma$ in such a way that the disk $B_i$ whose interior is removed from each $\Sigma_i$ intersects $D_i$ in a single arc with no crossings.  (Different choices of $B_1$ and $B_2$, may well lead to diagrams that represent different links in $\Sigma\times I$.) 
There is then a separating circle $\gamma\subset\Sigma$ such that  $(\Sigma\cut \gamma,D\cut \gamma)$ is pairwise homeomorphic to 
\[\left((\Sigma_1\cut B_1)\sqcup(\Sigma_2\cut B_2),(D_1\cut B_1)\sqcup(D_2\cut B_2)\right),\] meaning that there is a homeomorphism 
\[f:\Sigma\cut \gamma\to (\Sigma_1\cut B_1)\sqcup(\Sigma_2\cut B_2)\]
 that restricts to a homeomorphism (respecting crossing information)
 \[D\cut \gamma\to (D_1\cut B_1)\sqcup(D_2\cut B_2).\]  
We sometimes specify this circle $\gamma$ in our notation by writing 
 \[(\Sigma_1,D_1)\#_\gamma(\Sigma_2,D_2)=(\Sigma,D).\] 
 Inversely, given a pair $(\Sigma,D)$, any separating circle $\gamma\subset \Sigma$ with $|\gamma\pitchfork D|=2$ determines (up to pairwise homeomorphism of the resulting factors) a decomposition $(\Sigma_1,D_1)\#_\gamma(\Sigma_2,D_2)=(\Sigma,D)$. We call this {\it local} if either $\Sigma_i$ is a 2-sphere, and we call it {\it trivial} if either $(\Sigma_i,D_i)$ is pairwise homeomorphic to $(S^2,\bigcirc)$. 
 
 We now turn from diagrams to links.
Given two pairs $(\Sigma_1,L_1)$ and $(\Sigma_2,L_2)$, choose respective diagrams $(\Sigma_1,D_1)$ and $(\Sigma_2,D_2)$, construct $(\Sigma_1,D_1)\#_\gamma(\Sigma_2,D_2)=(\Sigma,D)$ as above, and, letting $L$ denote the link in $\Sigma\times I$ represented by $D$, write  $(\Sigma_1,L_1)\#_\gamma(\Sigma_2,L_2)=(\Sigma,L)$, perhaps omitting ``$\gamma$''. This operation is called an {\bf annular connect sum}.  Figure \ref{Fi:ACS} shows an example. If either $\Sigma_i$ is a 2-sphere, we call this a {\bf local connect sum}, since it can also be realized as a pairwise connect sum  
\[(\Sigma_{3-i}\times I,L_{3-i})\#(S^3,L_i)=(\Sigma\times I,L).\]

\begin{figure}
\begin{center}
\labellist \small
\pinlabel{$\#$} at 222 65
\pinlabel{$=$} at 480 65
\endlabellist
\includegraphics[width=\textwidth]{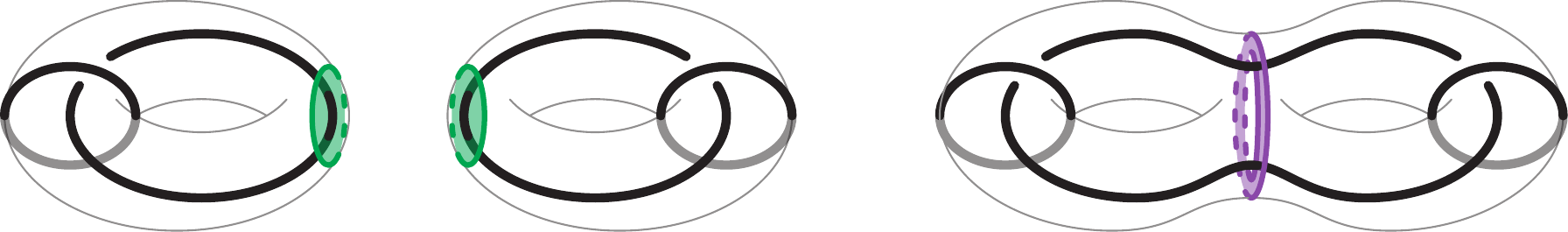}\\
\caption{An annular connect sum}
\label{Fi:ACS}
\end{center}
\end{figure}

Inversely, given a pair $(\Sigma,L)$, if a properly embedded annulus $A\subset \Sigma\times I$ with $|A\pitchfork L|=2$ has one boundary component on each of $\Sigma_\pm$, then cutting $\Sigma\times I$ along $A$ and gluing on two 3-dimensional 2-handles (each containing a properly embedded arc) in the natural way gives a(n annular connect sum) decomposition $(\Sigma,L)=(\Sigma_1,L_1)\#(\Sigma_2,L_2)$, which again is unique up to pairwise homeomorphism of the resulting factors. 

Several comments are in order, but we postpone them until \textsection\ref{S:Comments}, so that we can use the correspondences we will establish in \textsection\ref{S:LChar} to justify their details.  


\subsubsection{Bi-annular sums}\label{S:Bi}

\begin{figure}
\begin{center}
\labellist \small
\pinlabel{$\#$} at 237 65
\pinlabel{$=$} at 494 65
\endlabellist
\includegraphics[width=\textwidth]{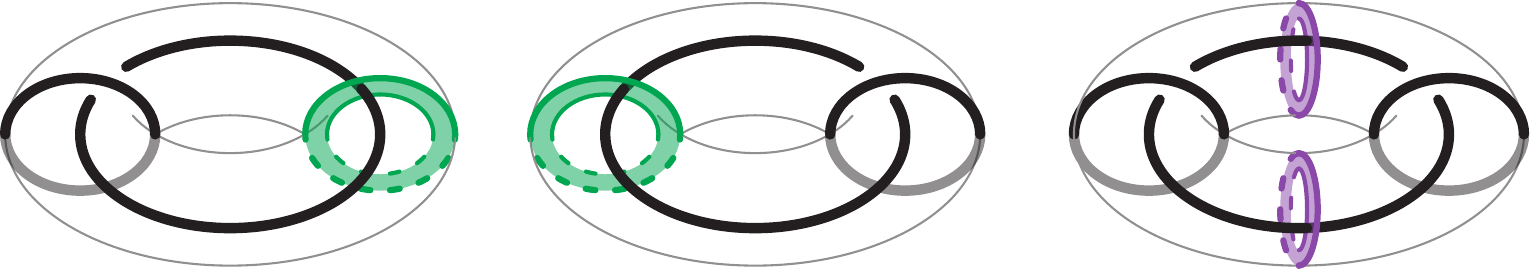}\\
\caption{A bi-annular sum}
\label{Fi:BiAnnular}
\end{center}
\end{figure}

Given two pairs $(\Sigma_1,L_1)$ and $(\Sigma_2,L_2)$, suppose that each $\Sigma_i\times I$ contains a properly embedded annulus $A_i$ that intersects $L_i$ transversally in a single point.  (Note that neither $(\Sigma_i,L_i)$ can then be checkerboard colorable.) Cut each $\Sigma_i\times I$ along $A_i$ and glue each resulting copy of $A_1$ to a copy of $A_2$ so that the endpoints of the $L_i$ align.   Call
\[(\Sigma,L)=((\Sigma_1\cut A_1)\cup(\Sigma_2\cut A_2),(L_1\cut A_1)\cup(L_2\cut A_2))\]
a {\bf bi-annular sum}. In an abuse of notation, write
\[(\Sigma,L)=(\Sigma_1,L_1)\#_{A_1\sqcup A_2}(\Sigma_2,L_2),\]
perhaps omitting the subscript ``$A_1\sqcup A_2$.'' If one properly isotopes $A_1\sqcup A_2$ in $\Sigma\times I$ to make each $A_i$ vertical, so that $A_i=\gamma_i\times I$ for some circle $\gamma_i\subset\Sigma$, extends to an ambient isotopy of $\Sigma\times I$, and perturbs the resulting embedding of $L$ so that it projects to a diagram $(\Sigma,D)$, then $(\Sigma,D)$ decomposes as a {\it diagrammatic bi-annular sum} along $\gamma_1\sqcup\gamma_2$.
Write
\[(\Sigma,D)=(\Sigma_1,D_1)\#_{\gamma_1\sqcup\gamma_2}(\Sigma_2,D_2),\]
possibly omitting the subscript $``\gamma_1\sqcup \gamma_2$.'' See Figure \ref{Fi:BiAnnular}, where each $A_i$ is green in $\Sigma_i\times I$ and each $\gamma_i\times I$ is purple in $\Sigma\times I$. 

Inversely, given a pair $(\Sigma,L)$, suppose $A_1,A_2\subset \Sigma\times I$ are disjoint properly embedded annuli, each with one boundary component in each $\Sigma_\pm$ with $|A_i\pitchfork L|=1$, such that $(\Sigma\times I)\cut (A_1\cup A_2)$ is disconnected. (Again, note that $(\Sigma,L)$ cannot be checkerboard colorable.) Cut $\Sigma\times I$ along $A_1\cup A_2$, and write $M_1$ and $M_2$ for the resulting connected components.
Each $\partial M_i$ contains a copy of $A_1$ and a copy of $A_2$. Glue them together so that the points of $L\cap A_i$ align.  This yields two pairs $(\Sigma_1,L_1)$ and $(\Sigma_2,L_2)$ that have $(\Sigma,L)$ as a bi-annular sum. 


The diagrammatic pairwise connect sum operation adapts naturally to virtual links and their diagrams: simply replace the pairs $(\Sigma_1,D_1)$, $(\Sigma_2,D_2)$, and $(\Sigma,D)$ in \textsection\ref{S:Sums} with virtual link diagrams $V_1$, $V_2$, and $V$ on $S^2$, representing virtual links $K_1$, $K_2$, and $K$, and write $(S^2,V_1)\#(S^2,V_2)=(S^2,V)$ or just $V_1\# V_2=V$, and write $K_1\#K_2=K$. Say that $V=V_1\#V_2$  is {\it nontrivial} if both $V_1$ and $V_2$ have classical crossings.




\subsection{Definitions of primeness}

\begin{definition}\label{D:VPrime}
A virtual link diagram $V $ is {\bf prime} if it has no nontrivial connect sum decomposition.
\end{definition}

\begin{rem}
Given a virtual knot diagram $V$ that comes from a Gauss code $G$, $V$ is nonprime if and only if, after some cyclic permutation, $G$ has the form $(a_1,\hdots, a_k,b_1,\hdots,b_\ell)$ where $k,\ell>0$ and $b_i\neq -a_j$ for all $i,j$ \cite{kauff98}.
\end{rem}

\begin{definition}\label{D:KPrime}
A virtual link $K\neq \bigcirc $ is {\bf prime} if for every diagram $V$ of $K$ and every decomposition $V=V_1\#V_2$, one diagram $V_i$ represents the classical unknot and the other represents $K$.
\end{definition}

Definition \ref{D:KPrime} is well motivated by the project of tabulating virtual knots and links, where, as in the classical setting, only primes are recorded.  Then every virtual link comes from the table via connect sum, although, unlike the classical setting, there are multiple (in fact infinitely many) ways to take the connect sum of two virtual links. 
Yet, Definition \ref{D:KPrime} is not quite traditional.  See the remark after Theorem \ref{T:01}

Call an annular or local connect sum decomposition $(\Sigma,L)=(\Sigma_1,L_1)\#(\Sigma_2,L_2)$ {\it trivial} if $(\Sigma_i\times I, L_i)$ is pairwise homeomorphic to $(S^2\times I,\bigcirc)$ for $i=1$ or $i=2$. Call a bi-annular sum decomposition $(\Sigma,L)=(\Sigma_1,L_1)\#_{A_1\sqcup A_2}(\Sigma_2,L_2)$ {\it trivial} if $A_1\cup A_2$ cuts off a thickened annulus from $\Sigma\times I$ that $L$ intersects in an unknotted arc.

\begin{rem}
Whereas every $(\Sigma,L)$ decomposes as a trivial local connect sum, $(\Sigma,L)$ decomposes as a trivial bi-annular sum if and only if $\Sigma\times I$ contains a properly embedded annulus $A$ with $|A\pitchfork L|=1$.
\end{rem}

 
\begin{definition}\label{D:LLPrime}
A nonstabilized link $L$ in a connected thickened surface $\Sigma\times I$ is {\bf locally prime} if it has no nontrivial local connect sum decomposition.\footnote{This is the traditional notion of primeness for $(\Sigma,L)$.  Chrisman calls such $(\Sigma,L)$ ``locally trivial" \cite{ch20}.}
\end{definition}

\begin{definition}\label{D:LPPrime}
A nonstabilized link $L$ in a connected thickened surface $\Sigma\times I$ is {\bf pairwise prime} if it has no nontrivial decomposition as an annular connect sum nor as a bi-annular sum.
\end{definition}

\begin{rem}
Definitions \ref{D:LLPrime} and \ref{D:LPPrime} extend easily to pairs $(\Sigma,L)$ that are stabilized (by choosing the unique nonstabilized representative of its stable equivalence class), or disconnected (by requiring each component to be locally prime). Theorem \ref{T:01} will extend in the same way to split links.
\end{rem}

We adapt Definitions \ref{D:LLPrime} and \ref{D:LPPrime} from links to diagrams:

\begin{definition}\label{D:DLPrime}
A nontrivial pair $(D,\Sigma)$ is {\bf weakly prime} if it admits no nontrivial {\it local} connect sum decomposition, i.e. if the only decompositions $(\Sigma,D)=(\Sigma_1,D_1)\#(S^2,D_2)$ are those with $(\Sigma_1,D_1)=S^2,\bigcirc)$ or $D_2=\bigcirc$. 
\end{definition}

This terminology, due to Howie--Purcell \cite{hp20}, highlights that this condition differs from one in \cite{oz06}, where Ozawa calls $(\Sigma, D)$ {\it strongly prime} if $\Sigma$ is connected and every curve on $\Sigma$ (not necessarily separating) that intersects $D$ generically in two points also bounds a disk in $\Sigma$ which contains no crossings of $D$. We also define:

\begin{definition}\label{D:LDPrime}
A nontrivial pair $(D,\Sigma)$ is {\bf pairwise prime} if it admits no nontrivial {\it pairwise} connect sum decomposition, i.e. if the only decompositions $(\Sigma,D)=(\Sigma_1,D_1)\#(\Sigma_2,D_2)$ are those with $(\Sigma_i,D_i)=S^2,\bigcirc)$ for $i=1$ or $i=2$. 
\end{definition}

\subsection{Characterizations}\label{S:LChar}

 \begin{theorem}\label{T:01}
Given a nonsplit virtual link $K$ and the corresponding nonstabilized link $L$ in a thickened surface $\Sigma\times I$, 
\begin{enumerate}
\item $(\Sigma,L)$ is locally prime if and only if $K$ admits no nontrivial connect sum decomposition $K=K_1\#K_2$ in which $K_1$ is a classical link and $g(K_2)=g(K)$, and
\item $K$ is prime if and only if $(\Sigma,L)$ is pairwise prime.
\end{enumerate}
\end{theorem}

\begin{rem}
Matveev {\it defines} prime virtual links this way \cite{mat} but does not mention (or prove) that this characterization coincides with Definition \ref{D:KPrime}. Part (2) of Theorem \ref{T:01} confirms that our natural definition is equivalent to Matveev's.
\end{rem}

\begin{rem}
It is easy to determine the genus of a virtual link $K$ with an alternating virtual diagram $V$: apply Correspondence \ref{C:Diagrams} to get an associated diagram $(D,\Sigma)$.  Then $g(K)=g(\Sigma)$, by Corollary \ref{C:bk36++}.
\end{rem}


\begin{proof}[Proof of Theorem \ref{T:01}]
The proof will be constructive, using just slight modifications of the procedures in Correspondence \ref{C:Diagrams}.  

Suppose $K$ has a virtual diagram $V$ with a diagrammatic connect sum decomposition $V=V_1\#_\gamma V_2$ in which, if either summand $V_i$ represents the trivial knot, then the other does not represent $K$. Choose an acceptable lasso $X$ for $V$ and isotope to minimize $|X\cap\gamma|$.  
Take a regular neighborhood $N$ of $V\cup X$ such that each arc of $\gamma\cap N$ contains at least one of the two points of $\gamma\cap V$.  Change the neighborhood of each virtual crossing as in Figure \ref{Fi:VtoA}. Call this $N'$.  Cap off each resulting boundary component abstractly with a disk to get a pair $(\Sigma,D)$, as in Correspondence \ref{C:Diagrams}. While capping off, extend $\gamma\cap N'$ to a closed 1-manifold $\gamma'\subset \Sigma$ by attaching an arc in each disk of $\Sigma\cut N'$ that contains endpoints of $\gamma\cap N'$. 
The diagram $(\Sigma,D)$ decomposes along $\gamma'$ as either an annular connect sum or a bi-annular sum.  The corresponding decomposition of $(\Sigma,L)$ is nontrivial, or else one $V_i$ would represent the trivial knot and the other would represent $K$. Thus, $K$ is not prime.

This proves on direction of (2). A similar argument proves the corresponding direction of (1): assume that $K$ has a virtual diagram $V=V_1\#_\gamma V_2$ where $V_1$ represents a classical link $K_1$ and $V_2$ represents a link $K_2$ with $g(K_2)=g(K)$.  Continue as before.  Then, in the nontrivial decomposition of $(\Sigma,L)$, the summand $(\Sigma_2,L_2)$ corresponding to $K_2$ has genus at least, hence equal to, $g(\Sigma)$.  Thus, the genus of the other summand is zero, and the decomposition is a local connect sum, and so $(\Sigma,L)$ is not locally prime.



For the converses, suppose first that $K$ is not prime.  Then $(\Sigma,L)$ decomposes nontrivially as either a bi-annular or annular sum, hence has a diagram $(\Sigma,D)$ that decomposes as either 
\[(\Sigma,D)=(\Sigma_1,D_1)\#_{\gamma_1\sqcup\gamma_2}(\Sigma_2,D_2)\text{ or }(\Sigma,D)=(\Sigma_1,D_1)\#_{\gamma}(\Sigma_2,D_2)\]
where neither $(\Sigma_i,D_i)$ is a diagram of the unknot on $S^2$. Let $K_i$ be the virtual link corresponding to the stable equivalence class of $(\Sigma_i,L_i)$, $i=1,2$.  If, say, $K_1$ is trivial, then $\Sigma_1$ has positive genus, so $g(\Sigma_2)<g(\Sigma)$, and since $(\Sigma,L)$ is nonstabilized, it follows that $K_2$ and $K$ are distinct. This completes the proof of (2).

If $(\Sigma,L)$ is not locally prime, then the same construction gives a nontrivial connect sum decomposition $K=K_1\#K_2$ in which $K_1$ is a classical link and $g(K_2)=g(K)$.
\end{proof}



\subsection{Commentary}\label{S:Comments}

Korablev--Matveev show that virtual knots decompose uniquely under connect sum:

\begin{theorem}[Theorem 12 of \cite{mat}, Theorem 2 of \cite{km11}]
Any virtual knot $K$ decomposes as a connect sum of prime and trivial virtual knots. The resulting prime summands depend only on $K$.
\end{theorem}

The following example of Kishino, however, shows that the converse is spectacularly false:

\begin{example}[\cite{kishino}]\label{Ex:Kishino}
Both summands $(\Sigma_i,L_i)$ in the annular connect sum $(\Sigma,L)=(\Sigma_1,L_1)\#(\Sigma_2,L_2)$ shown in Figure \ref{Fi:Kishino} are once-stabilizations of the unknot, but $(\Sigma,L)$ corresponds to a nontrivial virtual knot called the {\it Kishino knot}.
\end{example}

\begin{figure}
\begin{center}
\labellist\tiny
\pinlabel{$\#$} at 237 72
\pinlabel{$=$} at 494 72
\pinlabel{$\equiv$} at 970 72
\endlabellist
\includegraphics[width=\textwidth]{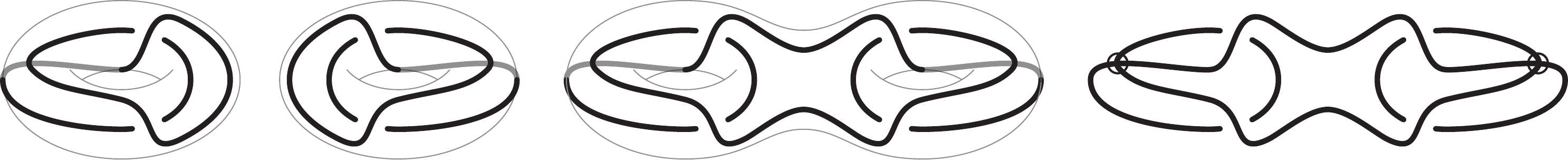}
\caption{The Kishino knot}
\label{Fi:Kishino}
\end{center}
\end{figure}

We note that, in general, if  $(\Sigma,L)=(\Sigma_1,L_1)\#(\Sigma_2,L_2)$ and both $(\Sigma_i,L_i)$ are nonstabilized, then $(\Sigma,L)$ is also nonstabilized. Example \ref{Ex:Kishino} shows that the converse of this, too, is false.  In particular, genera of virtual links are not additive under connected sum. Kauffman--Manturov show:

\begin{theorem}[Theorem 4 of \cite{km06}]
Let $K$ be a classical knot, let $K_1$ and $K_2$ be non-classical virtual knots. Then, for any connect sums $K\#K_1$ and $K_1\#K_2$, we have $g(K\#K_1)\geq g(K_1)$ and $g(K_1\#K_2)\geq g(K_1)+g(K_2)-1$.
 \end{theorem}

\begin{theorem}[Theorem 5 of \cite{km06}]\label{T:513}
If a connect sum $K_1\#K_2$ of virtual knots is trivial, then $K_1$ and $K_2$ are trivial.  
 \end{theorem}

We again note that Kishino's knot shows that the converse of Theorem \ref{T:513} is false and thus that connect sum is not a well-defined binary operation on virtual knots.

\subsection{De-summing lasso diagrams}\label{S:LassoSum}

\begin{definition}\label{D:DeSum}
Let $(X,V)$ be a lasso diagram, and let $\alpha\subset X$ be a properly embedded arc that intersects $V$ generically such that both disks of $X\cut\alpha$ contain crossings of $V$.  Write $m$ for the number of label-pairs on $\partial X$ that lie in opposite arcs of $\partial X\cut\partial\alpha$, and write $n=|\alpha\cap V|$.  If $m+n=2$, we call $\alpha$ a {\bf de-summing arc} for $(X,V)$. If $n=2$ and all labels on $\partial X$ lie in the same arc of $\partial X\cut\partial\alpha$, then there is a disk $U\subset X$ with $\partial U\cap V=\alpha\cap V=(\text{2 points})$, whose interior contains crossings of $V$; call $U$ a {\bf de-summing disk} for $(X,V)$.
\end{definition}

In Definition \ref{D:DeSum}, note that $m+n$ is always even (as $H_1(S^2;\Z/2)=0$), and $n$ is always positive when $(X,V)$ is acceptable.  


\begin{obs}\label{O:LDPrime}
A virtual link diagram $V$ which admits an acceptable lasso $X$ is prime if and only if the lasso diagram $(X,V\cap X)$ admits no de-summing arc.
\end{obs}

\begin{obs}\label{O:LDLPrime}
A virtual link diagram $V$ which admits an acceptable lasso $X$ is locally prime if and only if the lasso diagram $(X,V\cap X)$ admits no de-summing disk.
\end{obs}

\section{Alternating links in thickened surfaces}\label{S:ThickPrime}

The remainder of the paper is devoted to describing how to determine local and pairwise primeness by inspection in the alternating case.  This section does this for links in thickened surfaces, where local primeness was already characterized by \cite{adamsetal,hp20}.  The next section extends this section's results to virtual link diagrams and lasso diagrams.

 \subsection{Split alternating links in thickened surfaces}\label{S:SplitAlt}

For both technical and expository reasons, we examine splitness before addressing primeness.

\begin{fact}\label{F:Irr}
If $L_1$ is a link in a 3-manifold $M$ and $M\setminus L_1$ is irreducible, and if $L_2$ is a nonsplit link in $S^3$, then the complement of $L_1\#L_2$ in $M=M\#S^3$ is also irreducible.
\end{fact}

Fact \ref{F:Irr}, which follows from a standard innermost circle argument, extends to this fact about pairwise connect sums:

\begin{prop}\label{P:CSSplit}
Suppose $(\Sigma,L)=(\Sigma_1,L_1)\#_\gamma(\Sigma_2,L_2)$, where both $(\Sigma_i,L_i)$ are nonsplit. Then $(\Sigma,L)$ is also nonsplit.
\end{prop}

\begin{proof}[Proof of Proposition \ref{P:CSSplit}]
Note that $\Sigma_1$ and $\Sigma_2$ are connected, so $\Sigma$ is too.  In the case where either $\Sigma_i=S^2$, the proposition follows from Fact \ref{F:Irr}. We may thus assume that both $\Sigma_i$ have positive genus. As explained in \textsection\ref{S:Thick}, it follows that  both $(\Sigma_i\times I)\setminus L_i$ are irreducible. A standard innermost circle argument then shows that $(\Sigma\times I)\setminus L$ is irreducible.  

Assume for contradiction that $(\Sigma,L)$ is split.  Then there is a system $A$ of annuli, each with one boundary component on each of $\Sigma_\pm$, such that $(\Sigma\times I)\setminus A$ has two components, both of which intersect $L$. Among all choices for such $A$, choose one which minimizes $|A|$. Then $\partial A$ is essential in $\Sigma_\pm$, because $(\Sigma\times I)\setminus L$ is irreducible, so by minimality $A$ and $\gamma\times I$ intersect only in essential circles. Therefore, each component of $A$ is isotopic rel boundary in $\Sigma\times I$ to $\gamma\times I$.  Yet, $A\cap L=\varnothing$, so $L$ can be isotoped to be disjoint from $\gamma\times I$.  This contradicts the assumption that both $(\Sigma_i\times I,L_i)$ are nonsplit.
 \end{proof}

We use Menasco's crossing ball technique to give a new proof of the following theorem of Ozawa \cite{oz06}, as the argument will adapt nicely to give a proof of Theorem \ref{T:ObviouslyPrimeEtc}. (Boden-Karimi give an alternate proof of this theorem in \cite{bk20}.)

\begin{theorem}[\cite{oz06,bk20}]\label{T:ObvSplit}
Let $D\subset \Sigma$ be a connected, cellular alternating diagram of a link $L\subset \Sigma\times I$. Then $(\Sigma,L)$ is nonsplit.
\end{theorem}

\begin{figure}
\begin{center}
\hfill\includegraphics[width=.4\textwidth]{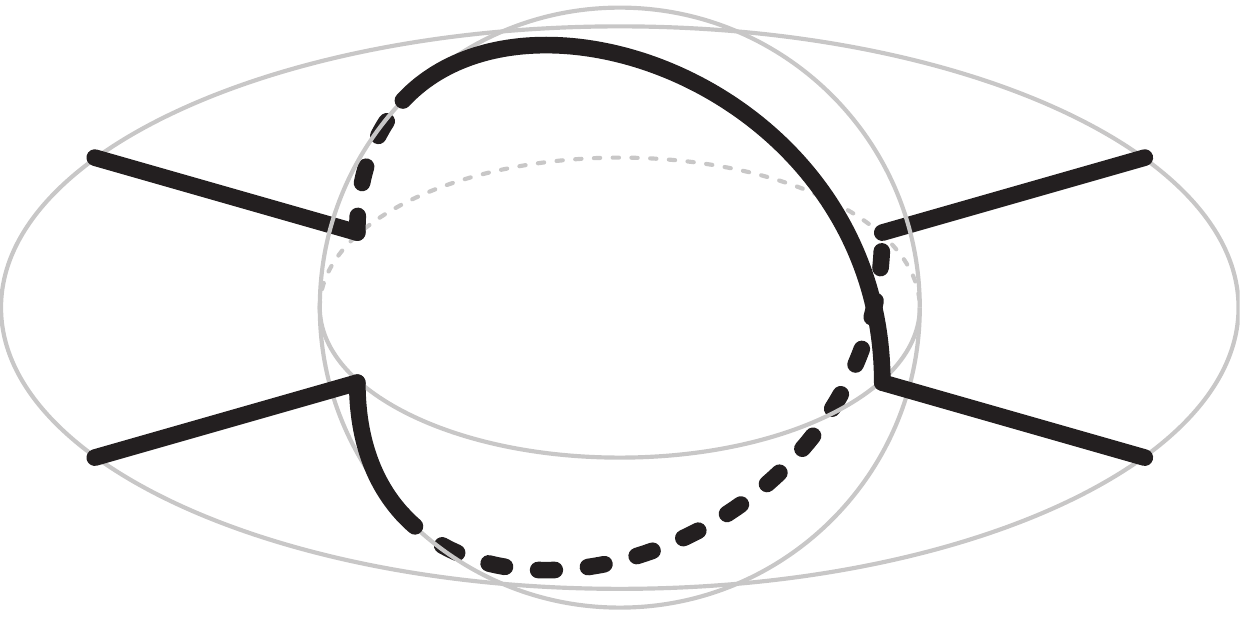}\hfill\includegraphics[width=.4\textwidth]{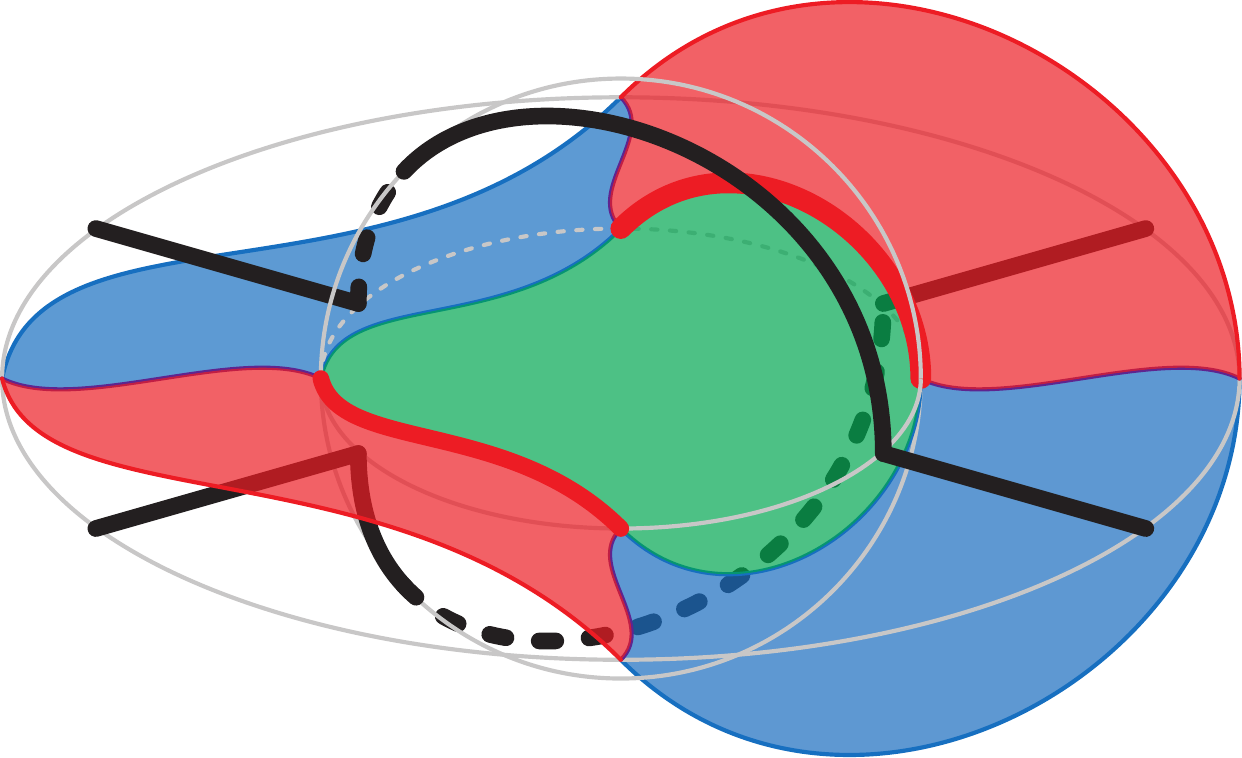}\hfill
\caption{A crossing bubble and a saddle}
\label{Fi:BubbleSaddle}
\end{center}
\end{figure}

\begin{proof}
By Proposition \ref{P:CSSplit}, it will suffice to prove this under the additional assumption that $(\Sigma,D)$ is pairwise prime. Implement the crossing ball setup a la Menasco by inserting a tiny ball $C_i$ centered at each crossing point of $D$ and pushing the two strands of $D$ near that crossing point to opposite hemispheres of that ball's boundary. Denoting $C=\bigcup_iC_i$, this gives an embedding of $L$ in $(\Sigma\setminus\text{int}(C))\cup\partial C$. See Figure \ref{Fi:BubbleSaddle}, left. Denote the two components of $(\Sigma\times I)\cut (\Sigma\cup C)$ by $H_+$ and $H_-$, and write each $\partial H_\pm\setminus \Sigma_\pm=S_\pm$.  %

Assume by way of contradiction that $(\Sigma,L)$ is split.  Then there is a system $A\subset(\Sigma\times I)\setminus L$ of disjoint, properly embedded annuli, each with one boundary component in each of $\Sigma_\pm$, which cuts $\Sigma\times I$ into two pieces, both of which intersect $L$. Of all choices for such $A$, choose one which lexicographically minimizes $\left(|A\cap C|,|A\setminus (\Sigma\cup C)|\right)$, provided $A\pitchfork \partial C$ and $A\pitchfork \Sigma\setminus \text{int}(C)$.  While fixing $A\cap (S_+\cup S_-)$, isotope $A\cut (S_+\cup S_-)$ so that all critical points of $\pi_\Sigma|_A$ lie in $A\cap (S_+\cup S_-)$.

Since $D$ is cellular, each component of $A$ must intersect $C$, and, by a standard argument, each component of $A\cap C$ is a saddle, as in Figure \ref{Fi:BubbleSaddle}, right \cite{men84}. Also, by minimality, no arc of $A\cap \Sigma\cut C$ is parallel in $\Sigma\cut C$ to $\partial C$.  Some component $V$ of $A\cut(\Sigma\cup C)$ must be a disk, due to euler characteristic considerations (for details, see the proof of Theorem \ref{T:ObviouslyPrimeEtc}).  Assume without loss of generality that $V\subset H_+$.  Denote $\gamma=\partial V\subset S_+$. Then $\gamma$ bounds a disk $U$ in $ S_+$, because $ S_+$, being isotopic to $\Sigma$, is incompressible in $\Sigma\times I$; by passing to an innermost circle of $A\cap S_+$ in $U$, we may assume that the interior of $U$ is disjoint from $A$.  

\begin{figure}
\begin{center}
\labellist
\small\hair 4pt
\pinlabel {$\boldsymbol{x_1}$} [c] at 225 120
\pinlabel {$\bs{x_2}$} [c] at 450 260
\pinlabel {$\bs{\brown{\delta_1}}$} [c] at 240 170
\pinlabel {$\bs{\brown{\delta_2}}$} [c] at 370 240
\pinlabel {$\bs{\red{\alpha}}$} [c] at 190 400
\pinlabel {$\bs{\FG{\beta}}$} [c] at 320 195
\endlabellist
\includegraphics[height=.5\textwidth]{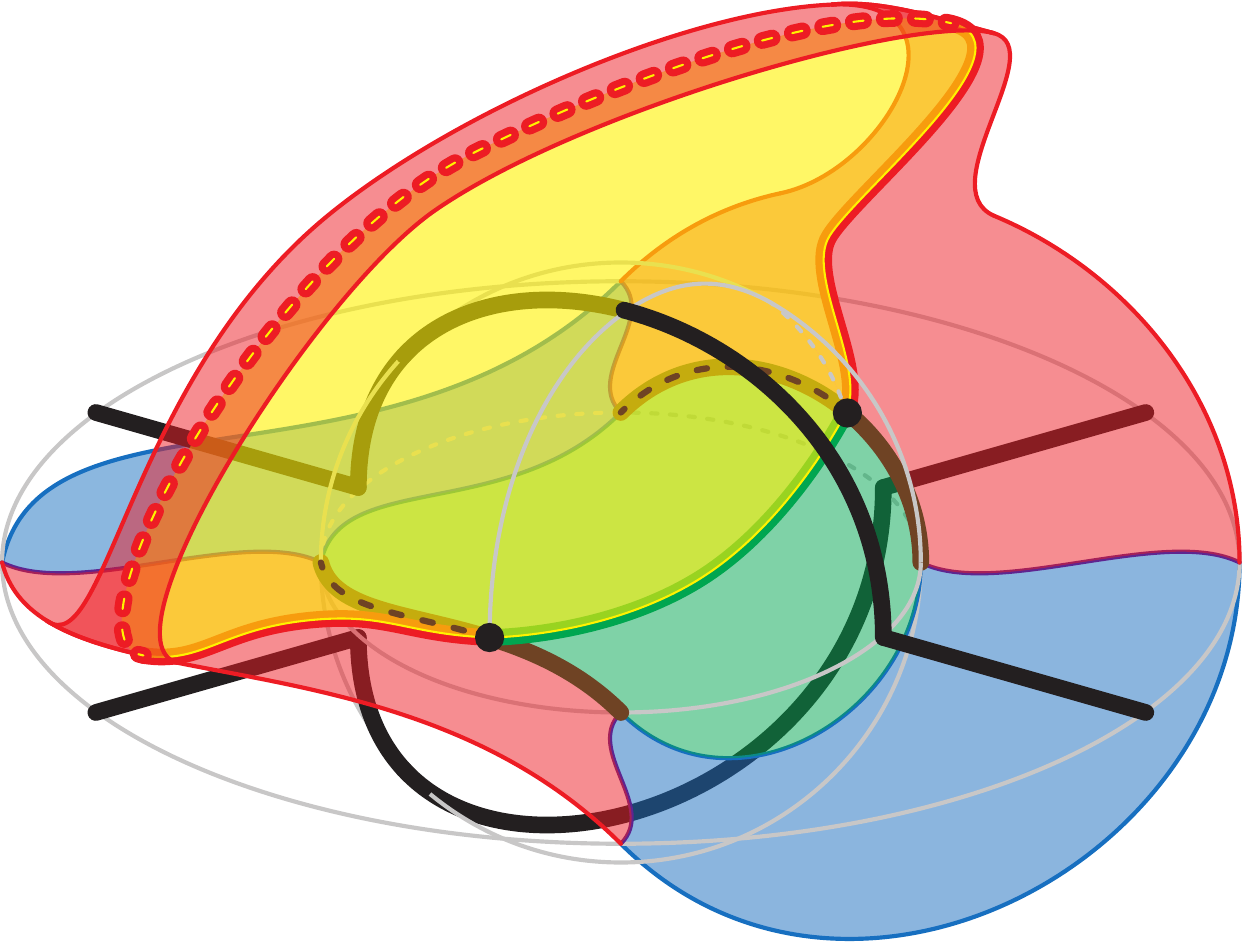}%
\caption{The arcs $\red{\delta_1},\red{\delta_2},\red{\alpha},$ and $\FG{\beta}$ and disks $\FG{S_0}$ and $X$ (yellow)  in the proofs of Theorems \ref{T:ObvSplit} and \ref{T:ObviouslyPrimeEtc}.}
\label{Fi:BubbleSaddleC}
\end{center}
\end{figure}
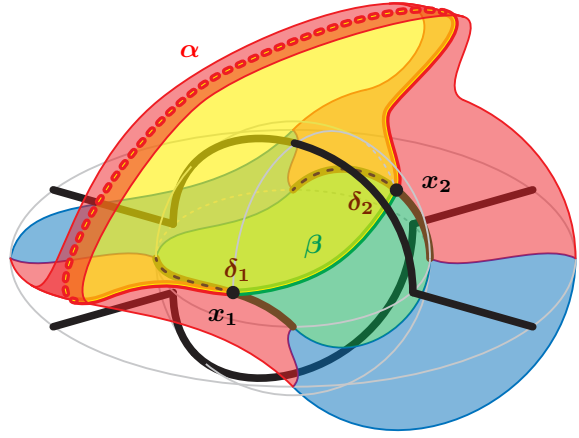

We claim that $\gamma$ {\it encloses} a saddle disk $S_0$ of $A\cap C$, meaning that $\pi_\Sigma(S_0)\subset \pi_\Sigma(U)$.  Such $S_0$ will yield a contradiction as follows.  The fact that $\text{int}(U)\cap A=\varnothing$ implies that both arcs $\delta_1$ and $\delta_2$ of $\partial S_0\cap S_+$ lie on $\gamma$.  Choose points $x_i\in\delta_i$, $i=1,2$, and construct arcs $\alpha\subset V$ and $\beta\subset S_i$ joining $x_1$ and $x_2$.  Then, as shown in Figure \ref{Fi:BubbleSaddleC}, $\alpha\cup\beta$ is a circle in $A$ which bounds a disk $X\subset (\Sigma\times I)\cut A$ with $|X\pitchfork L|=1$. This is impossible because $L$ is nullhomologous over $\Z/2$ in $\Sigma\times I$, hence in $(\Sigma\times I)\cut A$.

\begin{figure}
\begin{center}
\labellist
\small\hair 4pt
\pinlabel {$\red{\delta}$} [b] at 100 95
\pinlabel {$\violet{\varepsilon}$} [t] at 225 105
\pinlabel {$\red{\delta'}$} [b] at 365 68
\pinlabel {${C_2}$} [b] at 382 108
\pinlabel {${C_1}$} [b] at 82 34
\endlabellist
\includegraphics[width=.8\textwidth]{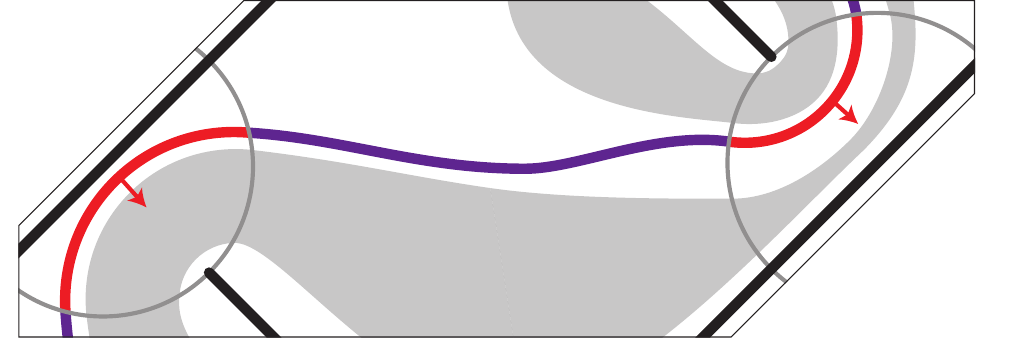}
\caption{Part of the circle $\gamma$. The arrows point into $U$. The shaded areas may contain more of $\gamma$.}
\label{Fi:SplitProp1E}
\end{center}
\end{figure}
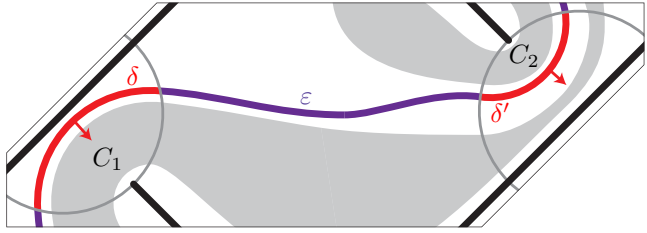

\begin{figure}
\begin{center}
\labellist
\tiny\hair 4pt
\pinlabel {$\red{\delta'}$} [b] at 92 98
\pinlabel {$\red{\delta''}$} [b] at 123 70
\endlabellist
\raisebox{-.075\textwidth}{\includegraphics[height=.15\textwidth]{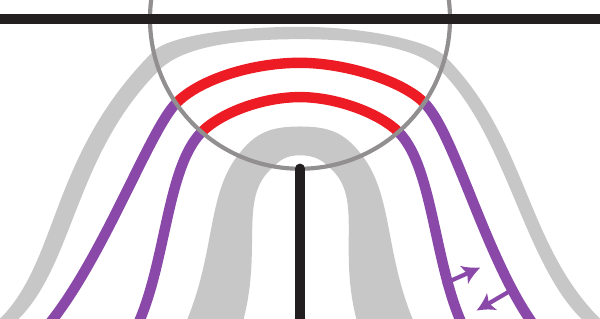}}{$~\to~$}
\raisebox{-.075\textwidth}{\includegraphics[height=.15\textwidth]{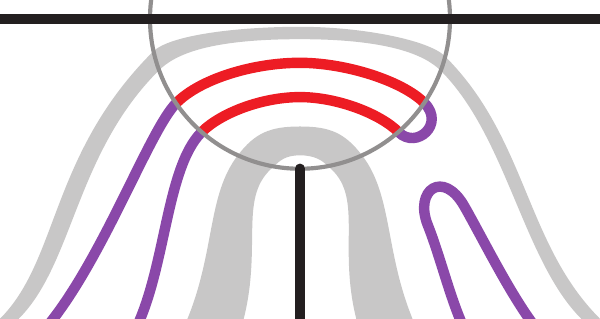}}{$~\to~$}
\raisebox{-.075\textwidth}{\includegraphics[height=.15\textwidth]{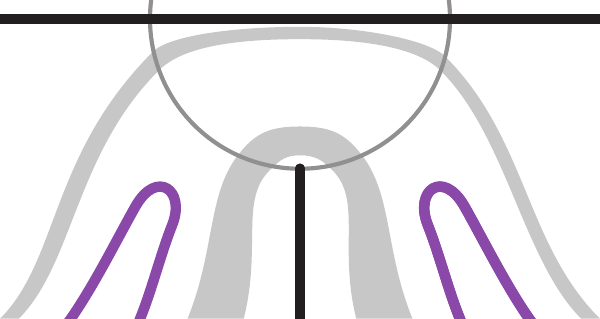}}
\caption{If $\delta'\cup\delta''\subset\gamma$ with arrows pointing into $U$, then this isotopy removes two saddles of $A\cap C$.}
\label{Fi:SplitProp1CD}
\end{center}
\end{figure}


It remains only to find such $S_0$.  Consider any crossing ball $C_1$ that $\gamma$ intersects, and take $\delta$ to be an arc of $\gamma\cap C_1$ closest to the overpass at $C_1$. If the overpass lies inside $U$, take $S_0$ to be the disk of $A\cap C_1$ incident to $\delta$.  (Recall that $A\cap \text{int}(U)=\varnothing$.) Otherwise, take $\varepsilon$ to be either arc of $\gamma\cut\partial C$ incident to $\delta$, denote the other crossing ball incident to $\varepsilon$ by $C_2$, and let $\delta'$ be the arc of $\gamma\cap \partial C_2$ incident to $\varepsilon$. See Figure \ref{Fi:SplitProp1E}. If no other arcs of $A\cap \partial C_2$ lie between $\delta'$ and the overpass at $C_2$, then this overpass lies inside $U$ (because $D$ is alternating), and we take $S_0$ to be the disk of $A\cap C_2$ incident to $\delta'$.  Otherwise, $\delta'$ is isotopic rel endpoints in $\partial C_2\cap U$ to an arc $\delta''$ of $\gamma\cap\partial C$, but as shown in  Figure \ref{Fi:SplitProp1CD}  this allows an isotopy of $A$ that decreases $|A\cap C|$, contrary to assumption. 
\end{proof}

\subsection{Prime alternating links in thickened surfaces}\label{S:PPAlt}

\begin{theorem}\label{T:ObviouslyPrimeEtc}
Let $D\subset \Sigma$ be a cellular alternating diagram of a link $L\subset \Sigma\times I$.  If $(\Sigma,D)$ is pairwise
prime, then $(\Sigma,L)$ is pairwise
prime.
\end{theorem}

 \begin{proof}[Proof of Theorem \ref{T:ObviouslyPrimeEtc}]
Implement the crossing ball setup a la Menasco as in our proof of Theorem \ref{T:ObvSplit}, and assume by way of contradiction that $(\Sigma,L)$ is not pairwise prime. Then, since $L$ is nullhomologous over $\Z/2$, there is an annulus $A$ with one boundary component in each of $\Sigma_\pm$, such that $|A\pitchfork L|=2$ and $(\Sigma \times I)\cut A$ consists of two components, neither of which is a ball whose intersection with $L$ is an unknotted arc.
Of all possibilities for such $A$, choose one which lexicographically minimizes $(|A\cap C|,|A\setminus(\Sigma\cup C)|)$, provided $A\pitchfork\partial C$ and $A\pitchfork \Sigma\setminus\text{int}(C)$.   

Then $A\cap C$ is comprised of saddles, and each $A\cap H_\pm$  is comprised of one annulus $A_\pm$ and, perhaps, some disks.  Denote $s=|A\cap C|$ and $f_\pm=|A\cap H_\pm|-1$, so that each $f_\pm$ is the number of {\it disks} of $A\cap H_\pm$.  The crossing ball structure induces a cell decomposition of $A$ with:
\begin{itemize}
\item $4s+2$ 0-cells: the four points of the ``equators" $ S_+\cap  S_-\cap \partial C$ on the boundary of each saddle, and one 
point $x_\pm$ on each of $\partial A\cap\Sigma_\pm$;
\item $6s+4$ 1-cells:
the arcs of $A\cap S_+\cap\partial C$, $A\cap  S_-\cap\partial C$, and $A\cap S_+\cap S_-$ ($2s$ of each),
the two arcs of $\partial A\setminus \{x_+,x_-\}$, and
an arc in each $A_\pm$ from $x_\pm$ to any 0-cell on $\partial A_\pm\cap S_\pm$; and
\item $s+f_++f_-+2$ 2-cells.
\end{itemize}
Hence:
\begin{align*} 0&=(4s+2)-(6s+4)+(s+f_++f_-+2)\\s&=f_++f_-.\end{align*}

Because $(\Sigma,D)$ is prime, $A$ must intersect $C$, and so $f_++f_-=s>0$.  Assume without loss of generality that $f_+>0$.  Then, arguing as in our proof of Theorem \ref{T:ObvSplit}, there is a circle $\gamma\subset A\cap  S_+$ which bounds disks $V\subset A\cap H_+$ and $U\subset  S_+\cut A$.

Any saddle disk $S_0$ of $A\cap C$ enclosed by $\gamma$ intersects $\gamma$ exactly once, or else, arguing as in our proof of Theorem \ref{T:ObvSplit}, there would be a disk $X\subset (\Sigma\times I)\cut A$ with $X\cap A=\partial X$ and $|X\pitchfork L|=1$ as shown in Figure \ref{Fi:BubbleSaddleC}, which now contradicts the minimality of $|A\cap C|$. This and the simplifying move shown in Figure \ref{Fi:SplitProp1CD} imply that no disk of $\partial C\cap S_+\cut L$ intersects $\gamma$ more than once. These two observations and the fact that $D$ is alternating imply that each arc of $\gamma\cut \partial C$ must intersect $L$. Moreover, there must be at least two arcs of $\gamma\cut\partial C$, or else $\gamma$ would appear as highlighted left in Figure \ref{Fi:PrimeSaddle}, allowing an isotopy of $A$ which decreases $|A\cap C|$, contrary to assumption. Note that the dark gray disk in the figure must contain only a crossingless arc of $D$, because $(\Sigma,D)$ is pairwise (hence locally) prime.

\begin{figure}
\begin{center}
\includegraphics[height=.3\textwidth]{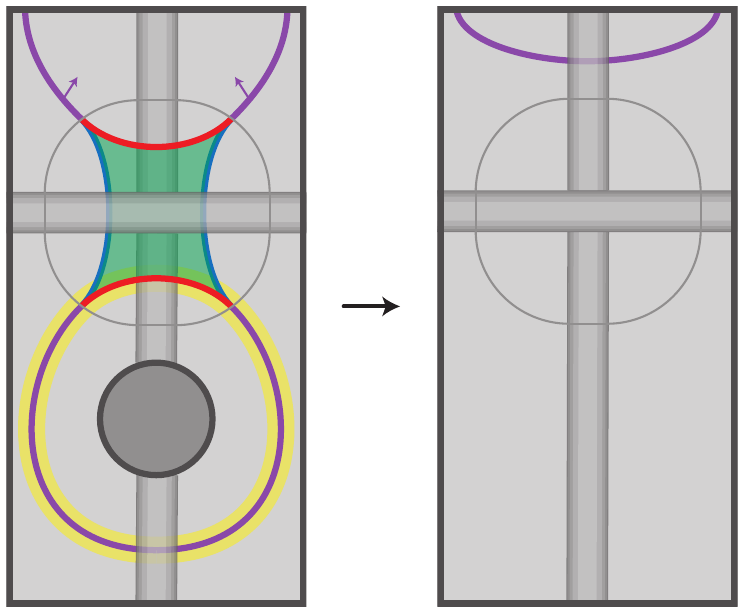}
\hfill
\includegraphics[height=.3\textwidth]{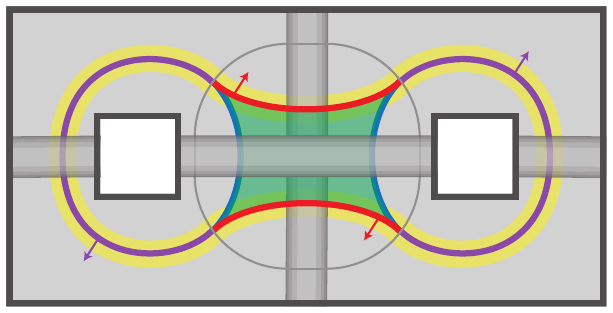}
\caption{The circle $\gamma$ cannot be as highlighted in yellow. Arrows point into $U$.}
\label{Fi:PrimeSaddle}
\end{center}
\end{figure}

Therefore, $\gamma$ contains both points of $A\cap L$.  Write $\gamma=\delta_1\cup\delta_2\cup\varepsilon_1\cup\varepsilon_2$, where each $\delta_i$ lies on the boundary of a crossing ball $C_i$ and each $\varepsilon_i\subset S_+\cut \partial C$.
In fact, $C_1\neq C_2$.  Figure \ref{Fi:PrimeSaddle}, right, shows why: if either white box in the figure is a disk in $\Sigma$, then we get  a contradiction as in Figure \ref{Fi:BubbleSaddleC}, and if not then $A\cap S_-$ contains two inessential circles, whereas only one is permitted.  Thus, $\gamma$ appears as in Figure \ref{Fi:PrimeCould}.

The preceding arguments regarding $\gamma$ apply to any innermost circle of $A\cap S_+$ that bounds a disk in $A$. Therefore, $\gamma$ is the only such circle; the circles of $(A\setminus A_+)\cap S_+$ are nested. Likewise, any innermost circle of $A\cap S_-$ that bounds a disk in $A$ would, like $\gamma$, need to contain just two arcs of $A\cap S_+\cap S_-$, each containing a point of $A\cap L$.  Yet, Figure \ref{Fi:PrimeCould} reveals that this is impossible, so $A$ intersects $H_-$ only in the annulus $A_-$. Thus, $f_-=0$ and $f_+=s\geq 2$.



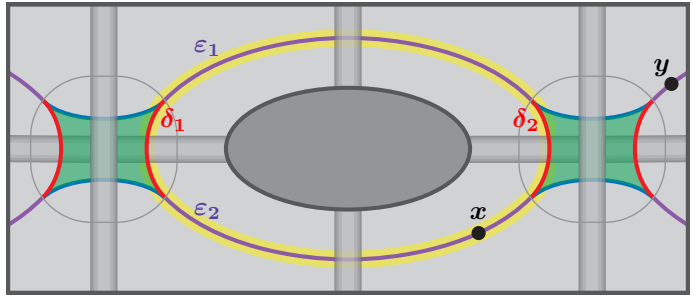
\begin{figure}
\begin{center}
\labellist
\small\hair 4pt
\pinlabel {$\bs{x}$} [b] at 350 50
\pinlabel {$\bs{y}$} [b] at 485 155
\pinlabel {$\violet{\bs{\varepsilon_1}}$} [b] at 150 170
\pinlabel {$\violet{\bs{\varepsilon_2}}$} [b] at 150 50
\pinlabel {$\red{\bs{\delta_1}}$} [b] at 125 115
\pinlabel {$\red{\bs{\delta_2}}$} [b] at 385 115
\endlabellist
\includegraphics[width=.8\textwidth]{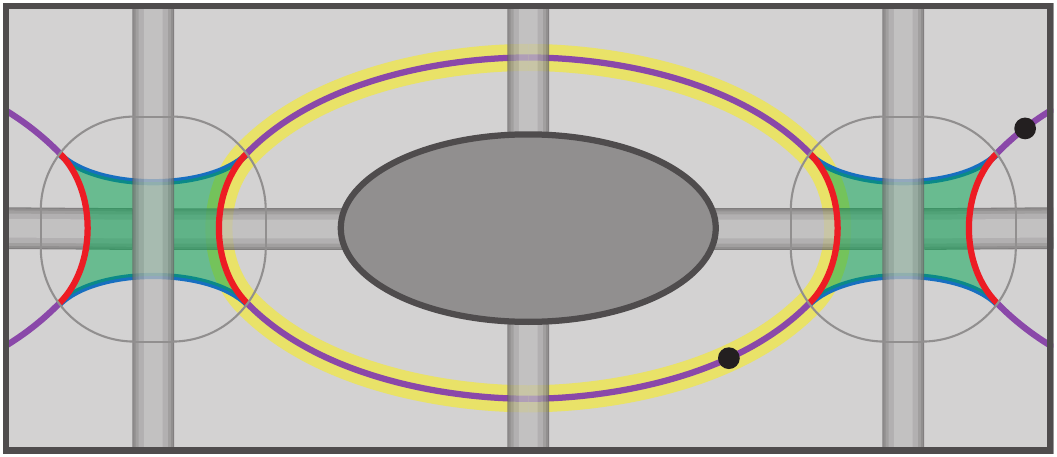}
\caption{The circle $\gamma$ (highlighted yellow) from the proof of Theorem \ref{T:ObviouslyPrimeEtc}.}
\label{Fi:PrimeCould}
\end{center}
\end{figure}

Denote the circles of $A\cap S_+$ by $\gamma_1=\gamma,\gamma_2,\hdots,\gamma_s,\gamma_{s+1}=\partial A_+\cap S_+$, such that, for each $i=1,\hdots,s$, $\gamma_i$ and $\gamma_{i+1}$ co-bound an annulus of $S_+\cut A$.
Observe that, for each $i=2,\hdots,s$, $\gamma_i$ abuts exactly four (distinct) saddles, which lie to opposite sides of $\gamma_i$: this follows by induction on $i$ as $\gamma_1$ is incident to exactly two saddles, $A\cap (H_+\cup C)$ is connected, and $D$ is alternating.  It follows that $\gamma_{s+1}$ abuts three saddles: two once and one twice. Thus, $|\gamma_1\cap C|=2$ and $|\gamma_i\cap C|=4$ for $i=2,\hdots, s+1$. Hence, there are $2s+1$ saddles: two incident to $\gamma_i$ and $\gamma_{i+1}$ for each $i=1,\hdots,s$ and one incident just to $\gamma_{s+1}$.  But the number of saddles also equals $s$, which is at least 2.  Contradiction.
%
\end{proof}

It is straightforward to adapt the preceding proof so as to recover:

\begin{theorem}[\cite{adamsetal,hp20}]\label{T:Local}
Let $D\subset \Sigma$ be a cellular alternating diagram representing a link $L\subset \Sigma\times I$.  If $(\Sigma,D)$ is locally prime, then $(\Sigma,L)$ is locally prime.
\end{theorem}

\section{Alternating virtual links}\label{S:Virtual}

\subsection{Nugatory crossings}\label{S:Nug}

The term ``nontrivial connect sum decomposition" carries quite different meanings for virtual links versus their diagrams: for virtual links, ``nontriviality" is about the types of links represented by the summands, whereas for diagrams it is about the presence of classical crossings in the summands.  In the alternating case, we will see that these notions coincide, but only after we account for ``nugatory" crossings, ones that can be removed by R1 moves 
after de-summing.  There are two types of such crossings: those that can also be removed (by R1 moves and flypes) before de-summing, and those that can be removed only after de-summing.

\begin{definition}\label{D:NugD}
A crossing $c$ in a pair $(\Sigma,D)$ is:
\begin{itemize}
\item {\bf nugatory} if there is a circle $\gamma\subset\Sigma$ with $\gamma\cap D=\{c\}$;
\item {\bf removably nugatory} if some disk $U\subset \Sigma$ has $\partial U\cap D=\{c\}$.
\end{itemize} 
\end{definition}

\begin{definition}\label{D:NugX}
A crossing $c$ in an acceptable lasso diagram $(X,V)$ is:
\begin{itemize}
\item {\bf nugatory} if some de-summing arc $\alpha\subset X$ has $\alpha\cap V=\{c\}$;
\item {\bf removably nugatory} if some disk $U\subset X$ has $\partial U\cap V=\{c\}$.
\end{itemize} 
\end{definition}

\begin{prop}\label{P:Nug}
Let $V$ be a virtual link diagram which admits an acceptable lasso $X$, let $(\Sigma,D)$ be the corresponding pair under Correspondence \ref{C:Diagrams}, let $c\in V$ be a classical crossing, and let $c'$ be the corresponding crossing of $D$. Then $c$ is (resp. removably) nugatory if and only if $c'$ is (resp. removably) nugatory.  
\end{prop}

\begin{proof}
This follows from Constructions \ref{C:LassoVtoD} and \ref{C:LassoDtoV} and Correspondence \ref{C:Diagrams}.  We give the details for one direction of each claim; the other direction is similar.

Let $Y$ be the acceptable lasso for $(\Sigma,D)$ that comes from $X$ via Construction \ref{C:LassoVtoD}, so there is a pairwise homeomorphism $f:(X,V\cap X)\to(Y,D\cap Y)$ that 
sends $c$ to $c'$. 

If $c'$ is nugatory, then there is a circle $\gamma\subset \Sigma$ with $\gamma\cap D=\{c'\}$. Isotope $\gamma$ so that $\gamma\cap Y$ is an arc.  Then $\alpha=f^{-1}(\gamma\cap Y)$ is a de-summing arc for $(X,V\cap X)$ with $\alpha\cap V=\{c\}$.
 
If $c'$ is removably nugatory, then a disk $U\subset \Sigma$ has $\partial U\cap D=\{c'\}$. Isotope $U$ so that $U\subset Y$. (This is possible because $Y$ is acceptable.) Then $U'=f^{-1}(U)$ is a disk in $X$ with $\partial U'\cap V=\{c\}$.
\end{proof}

\begin{prop}\label{P:cV}
Let $V$ be a virtual link diagram, let $V_1,V_2\in[V]$ admit acceptable lassos $X_1,X_2$, and let $c_1\in V_1$ and $c_2\in V_2$ be classical crossings that correspond to the same classical crossing $c\in V$.  The crossing $c_1$ is (resp. removably) nugatory in $(X_1,V_1\cap X_1)$ if and only if $c_2$ is (resp. removably) nugatory in $(X_2,V_2\cap X_2)$.
\end{prop}


\begin{proof}
Let $(\Sigma,D)$ be the cellular diagram associated to $V$ by Correspondence \ref{C:Diagrams}, and for $i=1,2$ let $Y_i$ be the acceptable lasso for $(\Sigma,D)$ obtained by applying Construction \ref{C:LassoVtoD} to $V_i$ and $X_i$.  Let $c'\in D$ be crossing that corresponds to $c$ (and to $c_1$ and $c_2$).  Proposition \ref{P:Nug} implies that $c'$ is (resp. removably) nugatory if and only if $c_1$ is, and likewise for $c_2$.  Thus, $c_1$ is (resp. removably) nugatory if and only if $c_2$ is.
\end{proof}

\begin{definition}\label{D:NugV}
Let $c$ be a classical crossing in a nonsplit virtual link diagram $V$.  Choose some $V'\in[V]$ which admits an acceptable lasso $X$, and let $c'\in V'$ correspond to $c$.  We call $c$:
\begin{itemize}
\item {\bf nugatory} if $c'$ is nugatory in $(X,V'\cap X)$;
\item {\bf removably nugatory} if $c'$ is removably nugatory in $(X,V'\cap X)$.\end{itemize} 
\end{definition}

Note that Definition \ref{D:NugV} relies on Propositions \ref{P:VLassoAcc} and \ref{P:cV}. 

\begin{obs}\label{O:Nug}
Suppose $V=V_1\#V_2$.  Then $V$ has nugatory crossings if and only if $V_1$ or $V_2$ does.
\end{obs}

\begin{fact}[Corollaries 1 and 2 of \cite{oz06}]\label{F:Nontrivial}
Any alternating virtual diagram with a positive number of crossings, none of them nugatory, represents a nontrivial virtual link.
\end{fact}

\subsection{Final result}\label{S:Central}


Recall from Proposition \ref{P:VLassoAcc}, whose proof is constructive, that every nonsplit class $[V']$ contains a diagram $V$ that admits an acceptable lasso $X$, and recall from \textsection\ref{S:Nug} that one can determine by inspecting $(X,V\cap X)$ whether or not $V$ has nugatory or removably nugatory crossings. Recall also from Observations \ref{O:LDPrime}-\ref{O:LDLPrime} that one can determine by inspecting $(X,V\cap X)$ whether or not $V$ is prime or locally prime.  Our central result is that, in the alternating case, one can determine by inspecting $(X,V\cap X)$ whether or not $V$ represents a prime or locally prime virtual link:

\begin{theorem}\label{T:VSplitEtc}
Let $V'$ be an alternating 
diagram of a nonsplit virtual link $K$, and consider a diagram $V\in[V']$ which 
admits an acceptable lasso $X$. Assume that $V$ has at least one virtual crossing. Then:
\begin{enumerate}[label=(\arabic*)]
\item When $V'$ has no nugatory crossings, $K$ is prime if and only if $V'$ is prime.
\item When $V'$ has no removably nugatory crossings, $K$ is locally prime if and only if $V'$ is locally prime.
\end{enumerate}
\end{theorem}

\begin{proof}
Let $(\Sigma,D)$ and $Y$ be the cellular alternating diagram and acceptable lasso obtained from $V'$ and $X$ via Construction \ref{C:LassoVtoD}, so there is a pairwise homeomorphism $f:(X,V'\cap X)\to(Y,D\cap Y)$.

For (a), assume that $V'$ has no nugatory crossings.  If $V'$ is not prime, then it decomposes as $V_1\# V_2$ where both $V_i$ are alternating with at least one crossing and no nugatory crossings. Thus, by Fact \ref{F:Nontrivial}, both $V_i$ represent nontrivial virtual links, and so $K$ is not prime.  Conversely, if $V'$ is prime, then $(X,V'\cap X)$ admits no de-summing arc, by Observation \ref{O:LDPrime}.  It follows that $(\Sigma,D)$ is pairwise prime, or else $(Y,D\cap Y)$ would admit a de-summing arc, as would $(X,V'\cap X)$. Since $(\Sigma,D)$ is pairwise prime, Theorem \ref{T:ObviouslyPrimeEtc} implies that $(\Sigma,L)$ is also pairwise prime.  Therefore, $K$ is prime, by part (2) of Theorem \ref{T:01}.
The proof of (b) is similar. 
\end{proof}

\section{Conclusion}\label{S:Conclude}

In the alternating case, one key difference between local and pairwise primeness is that a virtual knot can be uniquely identified by its locally prime factors, but not necessarily by its pairwise prime factors. 
To see why, observe that, among the infinitely man distinct connect sums of a given virtual knot and classical knot, there is only one of minimal genus, and if the two summands are alternating, then the one of minimal genus is the only one that is alternating.  
The same comment does not hold, however, for connect sums of two alternating virtual knots: 
 the virtual flyping theorem \cite{virtual} 
 implies that there are infinitely many distinct ways to take a connect sum of any two non-classical alternating virtual links, while preserving the alternating and minimal genus conditions.

In fact, questions about the proper scope of that theorem provided the initial motivation that led to this paper.  Tait's classical Flyping Conjecture (first proven by Menasco--Thistlethwaite \cite{menthis93} and later proven geometrically by the author \cite{flyping}) asserts that all prime alternating diagrams of a given link are related by so-called flype moves. Without the hypotheses of primeness, the conjecture is false: 
given a nontrivial diagrammatic connect sum $D=D_1\#D_2$ of reduced alternating link diagrams on $S^2$, one can ``slide $D_2$ along $D_1$" to change $D$ into other reduced alternating diagrams of the same link, and in general this sort of slide move cannot be achieved by flyping.  The same reasoning applies to an alternating diagram $D$ on a surface $\Sigma$ if $(\Sigma,D)$ admits a local decomposition $(\Sigma,D)=(\Sigma,D_1)\#(S^2,D_2)$. Therefore, the virtual flyping theorem is false without an assumption of local primeness.   What about links $L\subset \Sigma\times I$ that are locally prime but not pairwise prime?  The results of this paper enable the author to prove that the virtual flyping theorem holds for these links.  See \cite{virtual} for details.

\end{document}